%% file: srsc1.tex
\documentclass[reqno,10pt]{amsart}
\input{preamble}

\usepackage[active]{srcltx}

\title[]{Relatively Very Free Curves and Rational Simple Connectedness}

\author[DeLand]{Matt DeLand}
\address{Department of Mathematics \\
  SUNY Stony Brook\\ Stony Brook, NY 11790}
\email{deland@math.sunysb.edu}

\date{\today}
\begin{document}


\begin{abstract}
Given a morphism between smooth projective varieties $f: W \rightarrow X$, we study whether $f$-relatively free rational curves imply the existence of $f$-relatively very free rational curves.  The answer is shown to be positive when the fibers of the map $f$ have Picard number $1$ and a further smoothness assumption is imposed.  The main application is when $X \subset \PP^n$ is a smooth complete intersection of type $(d_1, \ldots, d_c)$ and $\sum d_i^2 \leq n$.  In this case, we take $W$ to be the space of pointed lines contained in $X$ and the positive answer to the question implies that $X$ contains very twisting ruled surfaces and is strongly rationally simply connected.  If the fibers of a smooth family of varieties over a $2$-dimensional base satisfy these conditions and the Brauer obstruction vanishes, then the family has a rational section (see \cite{dJHS}).
\end{abstract}


\maketitle

\section{Introduction}
In what follows, we fix the ground field to be the complex numbers.  The first part of this paper (through Section~\ref{section_geometry}) addresses the following general setup.

\begin{setup}\label{setup_main}
Let $f: W \rightarrow X$ be a surjective, proper morphism between irreducible, smooth, projective varieties of dimensions $m$ and $n$ respectively.  We assume that for a general $x \in X$, $f^{-1}(x)$ has Picard number $1$ (in particular, it is connected).  Suppose we have a map $\phi: \PP^1 \rightarrow W$ such that
\begin{itemize}
\item (Hypothesis A) The bundle $\phi^*f^*T_X$ is globally generated.  
\item (Hypothesis B) The sheaf $\phi^*T_f$ is locally free, globally generated, and has positive degree.
\end{itemize}
\end{setup}

The goal of this work is to identify additional (and hopefully verifiable) hypothesis under which there exists a map $\phi': \PP^1 \rightarrow W$ such that Hypothesis A is satisfied for $\phi'$ and such that $\phi'^*T_f$ is locally free and ample.

\begin{rmk}\label{rmk_general_case}
Recall that $T_f$ is defined to be the kernel of the differential $T_W \rightarrow f^*T_X$.  Requiring that $\phi^*T_f$ is locally free is equivalent to saying that the image of $\phi$ is contained in the smooth locus of $f$.
\end{rmk}

\begin{defn}
Let $f: W \rightarrow X$ be a map of varieties and $\phi: \PP^1 \rightarrow W$ be a map whose image is contained in the smooth locus of $f$.  Such a rational curve is called \textit{relatively free} if $\phi^*T_f$ is globally generated with positive degree and is called \textit{relatively very free} if $\phi^* T_f$ is, in addition, ample.
\end{defn}

The reason that we are interested in producing relatively very free curves has to do with producing rational sections of fibrations over surfaces.  Recall that when $f: \mc{X} \rightarrow C$ is a flat, proper, generically smooth morphism to a curve whose geometric generic fiber is rationally connected, then $f$ admits a section.  This is the statement of the Graber, Harris, Starr theorem; see \cite{GraberHarrisStarr}.  When the base has dimension $2$, there has been recent progress in identifying conditions on the geometric fiber of $f: \mc{X} \rightarrow S$ (a flat, proper, generically smooth morphism to a surface) which will guarantee the existence of a rational section.

\begin{thm}(Corollary 1.1, \cite{dJHS})\label{thm_sections_exist}
Let $f: X \rightarrow S$ be a flat, proper morphism to a smooth, irreducible, projective surface.  Assume there is an open $U \subset S$ whose complement has codimension two such that $f^{-1}(U)$ is smooth and $f^{-1}(U) \rightarrow U$ has geometrically irreducible fibers.  Let $\mc{L}$ be an $f$-ample invertible sheaf on $f^{-1}(U)$.  Assume that the geometric generic fiber of $f$ is rationally simply connected by chains of free lines and admits a very twisting surface.  Then there is a rational section of $f$.
\end{thm}

We refer to their paper for a thorough discussion of each assumption.  Here we remark that very few general classes of varieties satisfying these conditions are known.  In the known examples, it seems to be the case that the existence of a ``very twisting surface" is often the most difficult of these conditions to verify.  We refer to Section~\ref{section_twisting} and Appendix~\ref{app-twisting-surfaces} for a discussion of twisting  and very twisting (or $2$-twisting) surfaces.  The definition we use in this paper is slightly different from the definition used in \cite{dJHS} (though the two can be reconciled).  

Let $X \subset \PP^n$ be a smooth projective variety and let $F_1(X)$ be the Fano scheme of pointed lines on $X$, which we will suppose is smooth and which will play the role of $W$ in the main setup.  Let $f: F_1(X) \rightarrow X$ denote the evaluation morphism.  Roughly, a relatively free morphism $\phi: \PP^1 \rightarrow F_1(X)$ which satisfies hypotheses A and B above will correspond to a $1$-twisting surface on $X$, and a relatively very free morphism will correspond to a $2$-twisting surface on $X$.  It is in this sense that we are interested in a general method for producing $2$-twisting surfaces on $X$ when we know that $1$-twisting surfaces exist.  We now state our general theorem which will be better understood after setting up the notation in Section~\ref{section_geometry}.

\begin{introtheorem}  (See Theorem~\ref{thm_rel_vf})
Assume that we are in the setting of Setup~\ref{setup_main}.  Let $M$ be the unique component of the stable mapping space containing $\phi$ and let $M_2$ be the unique component of the two pointed stable mapping space dominating $M$.  Let $\phi': (\PP^1, p, q) \rightarrow W$ be a general map in $M_2$.  If for any $z \in f^{-1}f(\phi(q))$ and any map $(\psi, p', q') \in M_2$ with $\psi(p') = \phi(p)$ and $\psi(q') = z$, the map $\psi$ is unobstructed for the second evaluation morphism, then there is a relatively very free morphism to $W$.
\end{introtheorem}

The map $\phi$ as in Setup~\ref{setup_main} satisfies $\phi^*T_f = \mc{E} \oplus \OO^{a}$ for some number $a \geq 0$ and some ample vector bundle $\mc{E}$ on $\PP^1$.  Informally, we will show that by gluing together curves in $M$ and deforming them to a map $\phi'$, the rank of the ample subbundle of $\phi'^*T_f$ will increase.  The original map $\phi$ defines a sequence of components $M_i$ of maps, and by considering this sequence we will produce a naturally defined, rational foliation $W \dashrightarrow Y$ over $X$ such that relatively free curves are contracted.  For a more precise statement, see Section~\ref{section_foliation} and Theorem~\ref{thm-foliation}.  In Section~\ref{section_geometry} we will show that under the hypotheses, there are enough curves contracted by the foliation to conclude that the relatively free curves ``point in all directions", that $Y = X$, so that there is a relatively very free curve.  (See Theorem~\ref{thm_rel_vf}).

The main application of this result will be to smooth complete intersections in projective spaces, see Section~\ref{section_ci}.  To establish the existence of $1$-twisting surfaces, we will use a theorem of de Jong and Starr, see \cite{dJS}.  We will then be able to verify the existence of $2$-twisting surfaces in the following degree ranges.  See Corollary~\ref{cor_results} for a stronger statement than what follows.

\begin{thm}
Let $X \subset \PP^n$ be a smooth complete intersection of type $(d_1, \ldots, d_c)$ such that $\sum d_i^2 \leq n$ and such that the Fano variety of lines on $X$ is smooth.  Then $X$ contains a $2$-twisting surface.
\end{thm}

In Section~\ref{section_srsc}, we will outline the argument from \cite{dJS} to show that an $X$ which admits $1$ and $2$ twisting surfaces is strongly rationally simply connected.  Though this argument is not original work, it is difficult to point to the theorem in the paper cited, so we have found it useful to reproduce the outline of their method here.  This will imply that complete intersections in the above degree range are strongly rationally simply connected (see that section for a definition).  Finally we will show that these complete intersections satisfy the hypotheses of Corollary 1.1 of \cite{dJHS}.  Notice that the final remarks they make in Section 16 have been resolved by the work here - all of the twisting surfaces we produce are ruled by lines.  This is a long-winded way of saying that we can reprove Tsen's theorem, see \cite{Tsen-Lang}, which is certainly well known.  In the future, we hope to show that these methods apply in other situations to verify that other varieties are strongly rationally simply connected.

\begin{notat}
We freely use the notation developed in Koll\'ar's book \cite{Kollar-ratcurves}, especially Chapters II and IV concerning maps from $\PP^1$ to a variety.  For a coherent sheaf $\mc{F}$ on a variety $X$, we use $h^i(X, \mc{F})$ to denote the dimension of the cohomology group $H^i(X, \mc{F})$ as a complex vector space.  The canonical bundle on a smooth variety $X$ will be denoted by $K_X$.  Points of Kontsevich mapping spaces $\kspace{n}{X}{\beta}$ will be often referred to by $(\phi, p_1, \ldots, p_n)$ to denote a genus 0, stable map $\phi: C \rightarrow X$ with the $n$ marked points $p_i \in C_{smooth}$.
\end{notat}

\section*{Acknowledgements}
Showing that degree $3$ hypersurfaces in projective space contain $2$-twisting surfaces constituted a portion of the author's PhD. thesis.  He would like to acknowledge his advisor, Johan de Jong, for all of his help and encouragement.  The author would also like to thank Jason Starr and Carolina Aruajo for helpful conversations and comments during the preparation of this work.

\section{Preliminaries}\label{section_prelims}

In this section we collect some results and set up some notation which will be used during the course of the proof.  As every vector bundle on $\PP^1$ splits into the direct sum of line bundles, the goal of the next sections is to control this splitting type for a given vector bundle.  Let $E$ be a vector bundle on $\PP^1$.  Write $E = \oplus \OO(a_i)$ with $a_1 \leq a_2 \leq \cdots \leq a_n$.

\begin{defn}
With the notation as above, let $p$ be the smallest index such that $a_p > 0.$  We say that the \textit{positive part} of $E$ is the subbundle $Pos(E) = \oplus_{j \geq p} \OO(a_j) \subset E$.  We say that $E$ is \textit{almost balanced} if its splitting type satisfies $a_n - a_1 \leq 1$.
\end{defn}

\begin{lem}\label{lem_find_section}
Let $\pi: \mc{C} \rightarrow B$ be a flat family of genus $0$ curves over a one dimensional, irreducible, smooth, pointed base $0 \in B$.  Assume that the general fiber is a smooth rational curve and the central fiber $C_0$ is reducible with two components which meet at a single node.  Then, for $i = 1, 2$, there is a section of $\pi$, $\sigma_i$, such that $\sigma_i(B) \cdot C_i = 1$ and $\sigma_i(B) \cdot C_{3 - i} = 0$
\end{lem}

\begin{proof}
After blowing down the component $C_{3 - i}$ over $B$, this is a special case of the fact that weak approximation is satisfied at places of good reduction for rationally connected fibrations.  See Theorem $3$ of \cite{HassettTschinkel_WA_function_fields}
\end{proof}

We will make use of the following deformation-theoretic results.

Suppose that $\phi \in \kspace{0}{W}{\beta}$ is a point corresponding to a stable map $\phi: C \rightarrow W$.   Set $\alpha = f_*\beta \in H^2(X, \ZZ)$.  Define the stack $\mc{S}$ by the following fiber product diagram.

$$
\xymatrix{
\mc{S} \ar[r] \ar[d] & \kspace{0}{W}{\beta} \ar[d] \\
[f \circ \phi] \ar[r] & \kspace{0}{X}{\alpha}
}
$$

A $T$ point of $\mc{S}$ corresponds to a family of genus $0$ curves $\pi: \mc{C} \rightarrow T$ and a stable map $g: \mc{C} \rightarrow W$ such that $f \circ g = f \circ \phi$ up to stabilization.  Notice that in the situation of Setup~\ref{setup_main}, Hypothesis A implies that the class $\alpha$ is not zero.

\begin{lem}\label{lem_relative_def}
The tangent space to $\mc{S}$ at $\phi$ is identified with the vector space $Hom_C(\phi^* \Omega_{W/X}, \OO_C)$.  For any deformation of $\phi$ over a local Artin ring $A$ and for any surjection $B \rightarrow A$ of local Artin rings with square zero kernel, there is an element $ob \in Ext^1_C(\phi^* \Omega_{W/X}, \OO_C)$ whose vanishing is necessary and sufficient to guarantee that the given deformation lifts to $B$.

If $C$ maps into the smooth locus of $f$ and $H^1(C, \phi^*T_f) = 0$, then $\mc{S}$ is smooth of dimension $h^0(C, \phi^*T_f) = h^0(C, \phi^*T_W) - h^0(C, \phi^*f^*T_X)$ at $\phi$.
\end{lem}

We omit the proof and simply note that this follows from general deformation theory considerations.  See, for example, \cite{Illusie_complexe_cotangent} III.2.3.2, suitably interpreted.

We also have need to consider the similar pointed case.  Let $(p,\phi) \in \kspace{1}{W}{\beta}$ be a point corresponding to a pointed, stable map $\phi: C \rightarrow W$, $p \in C$.  Define $\mc{S}_1$ to be the fiber product of $[f \circ \phi, p]$ with $\kspace{0}{W}{\beta}$ over $\kspace{0}{X}{\alpha}$ exactly analogously to the above setup.  The following is an analog of the previous Lemma.  We again omit the proof.

\begin{lem}\label{lem_relative_def_fix_pt}
The tangent space to $\mc{S}_1$ at $(p, \phi)$ is identified with the vector space $Hom_C(\phi^*\Omega_{W/X}(p), \OO_C)$.  For any deformation of $(p,\phi)$ over a local Artin ring $A$ (subject to the above constraints) and for any surjection $B \rightarrow A$ of local Artin rings with square zero kernel, there is an element $ob \in Ext^1_C(\phi^* \Omega_{W/X}(p), \OO_C)$ whose vanishing is necessary and sufficient to guarantee that the given deformation lifts to $B$.

If $C$ maps into the smooth locus of $f$ and if $H^1(C, \phi^*T_f(-p)) = 0$, then $\mc{S}_1$ is smooth of dimension $h^0(C, \phi^*T_f(-p)) = h^0(C, \phi^*T_W(-p)) - h^0(C, \phi^*f^*T_X(-p))$ at $(p, \phi)$.
\end{lem}

\section{The Foliation}\label{section_foliation}
The first results of this section apply to the following general situation.  To be sure, we do not yet assume that we are in the setting of Setup~\ref{setup_main}.

\begin{sit}\label{gen_situation}
Let $V$ be a smooth, irreducible, projective variety and let $\phi: \PP^1 \rightarrow V$ be a free rational curve with $\phi^*T_V$ of degree larger than $2$.  Suppose $E$ is a vector bundle of rank $r$ on $V$ such that $\phi^*E$ is globally generated and has positive degree.  Set $\beta = \phi_*[\PP^1]$ in $A_1(V)$.
\end{sit}

\begin{lem}\label{lem_unique_component}
In the setting of Situation~\ref{gen_situation}, the map $\phi$ corresponds to a smooth point of $\kspace{0}{V}{\beta}$.  In particular, the corresponding point $[\phi]$ is contained in a unique component $M_1$ of this space.  A general point of $M_1$ corresponds to a free map from an irreducible genus $0$ curve such that $E$ pulls back to a globally generated bundle of positive degree.
\end{lem}
\begin{proof}
The obstruction space to lifting infinitesimal deformations of the map $\phi$ is identified with $H^1(\PP^1, \phi^*T_V)$.  By definition, this group is $0$ and the smoothness follows.  The second statement follows from the fact that being globally generated is an open condition in families of maps.  The degree of the bundle is constant in such families.
\end{proof}

We now define a sequence of components $M_i \subset \kspace{0}{V}{i\cdot\beta}$ indexed by $i \in \mathbb{N}$.  The component $M_1$ is the one whose existence is asserted in Lemma~\ref{lem_unique_component}.  There is a unique component $M_{1,1} \subset \kspace{1}{V}{\beta}$ dominating $M_1$.  As a general point $[\phi] \in M_1$ is a free curve on $V$, the evaluation map $ev: M_{1,1} \rightarrow V$ is surjective.

\begin{lem}\label{lem_fibers}
A general fiber of $ev: M_{1,1} \cap \kspacenb{1}{V}{\beta} \rightarrow V$ is smooth.  Such a fiber has dimension $h^0(\PP^1, \phi^*T_V(-p)) - h^0(\PP^1, T_{\PP^1}(-p))$.
\end{lem}
\begin{proof}
For a general point $v \in V$, let $(\phi: \PP^1 \rightarrow V, p \in \PP^1)$ be a pointed map such that $\phi(p) = v$. The map $\phi: \PP^1 \rightarrow V$ is a free curve.  The group $H^1(\PP^1, \phi^*T_V(-p))$ contains the obstructions to lifting infinitesimal deformations of the map $\phi$ which contain the point $p$.  This group is $0$ because $\phi$ is free and so $[\phi]$ is a smooth point of the fiber of the map $ev: M_{1,1} \rightarrow V$.  The dimension of the fiber at $\phi$ is then $h^0(\PP^1, \phi^*T_V(-p)) - h^0(\PP^1, T_{\PP^1}(-p))$.
\end{proof}

The existence of the component $M_1$ singles out a well-defined sequence of components.

\begin{constr}\label{construct_components}
Consider a nodal, arithmetic genus zero curve $C$ with components $C_i$ and a map $\phi: C \rightarrow V$.  If each component of $C$ is contracted by $\phi$ or else $\phi|_{C_i}:C_i \rightarrow V$ is a general point of $M_1$, then $\phi$ is a smooth point of $\kspace{0}{V}{k\cdot\beta}$ (here $k$ is the number of non-contracted components) and a general deformation of $\phi$ smooths all the nodes of $C$ (see \cite{Kollar-ratcurves} Theorem II.7.6).  Thus $\phi$ defines a unique irreducible component $M_k \subset \kspace{0}{V}{k \cdot \beta}$ such that $M_k \cap \kspacenb{0}{V}{k\cdot\beta}$ is not empty.
\end{constr}

This construction of the components $M_i$ is well defined in the following sense.

\begin{lem}\label{lem_well_defined_sequence}
Given two positive integers $a,b$, if $\phi_a: \PP^1 \rightarrow V$ is a general point of $M_a$ and $\phi_b: \PP^1 \rightarrow W$ is a general point of $M_b$ with $\phi_a(0) = \phi_b(0)$, then the unique map from $\phi_a \coprod \phi_b: \PP^1 \coprod_0 \PP^1 \rightarrow V$ is a smooth point of $M_{a+b}$.  Further, for any $j > 0$, a general point $[\phi_j] \in M_j$ is a free curve such that $\phi_j^*E$ is globally generated and has positive degree.
\end{lem}
\begin{proof}
The first statement follows from the fact that a general point of $M_a$ can, by definition, be degenerated to a map from a nodal curve $\phi': C \rightarrow V$ where $\phi'$ restricted to each of the $a$ components of $C$ is a free map contained in $M_1$.  The smoothness of the point $\phi_a \coprod \phi_b$ is straightforward.

The second statement follows by hypothesis for $j = 1$.  To verify the general case, consider a general smoothing of a map $\phi: C \rightarrow V$ with each irreducible component of $C$ mapping to $V$ via a general point of $M_1$ as above.  Such a smoothing looks like a family over a pointed base $\pi: \mc{C} \rightarrow (0 \in B)$ with a map $\Phi: \mc{C} \rightarrow V$ such that $\Phi$ restricts to $\phi$ on $\mc{C}_0$ and on a general fiber $\Phi_b: \mc{C}_b \rightarrow V$ is a general point of $M_j$.  The result follows because $\Phi^*E$ and $\Phi^*T_V$ are globally generated on the special fiber of $\pi$ and being globally generated on fibers is an open condition on the base $B$.  The degree of $\Phi^*E$ is constant on fibers.
\end{proof}

For each $[\phi] \in M_j$ which is a map from an irreducible curve to $V$, we may consider the splitting type of the bundle $\phi^*E$.  That is, we may write $\phi^*E$ as the direct sum of line bundles $\bigoplus \OO(a_i)$ with $a_1 \leq \ldots \leq a_r$.  On an open set in $M_j$, this splitting type is constant (again by the semi-continuity theorem) and we have $0 \leq a_1$ by Lemma~\ref{lem_well_defined_sequence}.

\begin{defn}\label{def_pos_rk}
For a $\phi \in M_j$ as above, we define $\it{Pos(\phi)}$ to be the largest rank ample sub-bundle of $\phi^*E$. Define $m_\phi$ to be the rank of $Pos(\phi)$.
\end{defn}

The following follows directly from the preceding Lemma and the above discussion.

\begin{lem}\label{lem_max_dim}
The integer $m_\phi$ is constant over an open set of $M_j$, denote this number by $m_j$; it satisfies $1 \leq m_j \leq r$.  The maximum is achieved if and only if $\phi^*E$ is ample.
\end{lem}

Denote the maximum value $m_{max} = \max_j\{m_j\}$, and the smallest $j$ which realizes this maximum by $j_{max}$.  We wish to find conditions guaranteeing that $m_{max} = r$.  If we can find two maps agreeing at a point whose positive pieces point in different directions at the point of intersection, then upon smoothing, the positive rank of the resulting map must increase.  This is made precise by the following.

\begin{prop}\label{prop-pos-subspace}
Let $\phi_1 \in M_{r_1}$, $\phi_2 \in M_{r_2}$ be maps from irreducible curves such that $\phi_1(0) = \phi_2(0)$.  Let $m' = \max(m_{\phi_1}, m_{\phi_2})$.  If $$\Span(Pos(\phi_1)|_0, Pos(\phi_2)|_0) \subset E_{\phi_i(0)}$$ has dimension greater than $m'$, then for a general deformation $\phi$ of $\phi_1 \cup \phi_2$, $m_\phi > m'$
\end{prop}

\begin{proof}
Denote the domain of $\phi_i$ by $C_i \cong \PP^1$.  Let $\pi: C \rightarrow B$ be a general deformation of $\phi_1 \cup \phi_2$.  This will be a diagram of the form
$$
\xymatrix{
C_1 \cup C_2 \ar[d] \ar[r] & C \ar[r]^\Phi \ar[d]^{\pi} & V \\
0 \ar[r] & B
}.
$$
We may assume that $\pi$ has two disjoint sections $D_1$ and $D_2$ meeting the central fiber only at a single point of $C_1$ and $C_2$ respectively; see Lemma~\ref{lem_find_section}.  Denote the restriction of $\Phi^*E(-D_1 - D_2)$ to $C$, respectively to $C_1$, $C_2$, by $F$, respectively by $F_1$, $F_2$.  Let $i_1: C_1 \rightarrow C_0$ and $i_2: C_2 \rightarrow C_0$ be the inclusion maps.  The short exact sequence
$$
0 \rightarrow F \rightarrow i_{1*}F_1 \oplus i_{2*}F_2 \rightarrow F|_0 \rightarrow 0
$$
gives rise to a piece of the long exact sequence in cohomology
$$
\xymatrix{H^0(C, F) \ar[r] & H^0(C_1, F_1) \oplus H^0(C_2, F_2) \ar[r]^-{a} & E|_0 \ar[r] & H^1(C, F) \ar[r] & 0}.
$$
Here the last $0$ follows because for each $\phi_i \in M_{r_i}$ the corresponding bundle $\phi_i^*E$ is semi-positive.   By assumption, the map $a$ has rank greater than $m'$ and so $h^1(C, F) < r - m'$.  By the semicontinuity theorem (\cite{Hartshorne-book} III.12.8), for a general $b \in B$, $h^1(C_b, \Phi^*E(-D_1 - D_2)|_{C_b}) < r - m'$ as well.  But on a general fiber, $\Phi^*E(-D_1 - D_2)|_{C_b} \cong \Phi_b^*E(-2)$ and moreover $E|_{C_b}$ is globally generated.  In other words, $\Phi_b^*E \cong \oplus \OO(a_i)$ with each $a_i \geq 0$ and so the number of $i$'s with $a_i > 0$ must be greater than $m'$, as claimed.
\end{proof}

We now define a subsheaf of $E$ which ``encodes" these positive subspaces.  For a general point $p \in V$, define $$\mc{D}(p) = \Span_{\phi \in M_j, \phi(0) = p} (Pos(\phi)|_0).$$  This definition makes sense on an open set $U$ of $V$ because a curve passing through a general point of $V$ is free and on a (possibly smaller) open set $U'$, the curves $\phi$ passing through $p \in U'$ will satisfy that $Pos(\phi)|_p$ has generic rank $m_j$.  The reader may be concerned that it seems that we must take the complement of countably many closed sets, but this is not the case because the maximum positive rank $m_{max}$ will be realized at some finite value of $j$.  Thus the above span need only take the values of $j$ less than or equal to $j_{max}$.

\begin{prop}\label{prop_glue_bundle}
There is a subsheaf $\mc{D}$ of $E$ and an open set $U \subset V$ such that for $p \in U$, $\mc{D}|_p = D(p)$.  In fact, $\mc{D}$ restricts to a vector bundle on some open set of $V$ whose complement has codimension at least $2$.
\end{prop}

\begin{proof}
A proof can be found in \cite{Mingmin_foliations}, Proposition 2.5.  Note that the proof given there is not phrased in this generality, but it readily extends.  We sketch the steps here.  Set $M = M_{j_{max}}$.  This is, for a general $[\phi] \in M$, the map corresponding to $\phi$ has maximal positive rank, $m_{max}$.  By shrinking $M$, we may assume that the splitting type for $\phi^*E$ is constant over $M$.  The universal map $\pi: M \times \PP^1 \rightarrow W$ is smooth, let $U$ be the image (which is open).  Denote the positive subbundle of $\pi^*E$ by $G$ and let $Z = M \times \PP^1$.  We may construct descent data on $G$ from the descent data on $\pi^*E$.  Indeed, for the scheme $Z \times_X Z$ with projections $p$ and $q$ to $Z$, there is an isomorphism $p^*G \cong q^*G$ because the positive subspaces are well defined.  That this isomorphism satisfies the cocycle condition follows from the same fact.  By faithfully flat descent, there is a subbundle $\mc{F}_U$ of $E$ over $U$, that pulls back to $G$ via $\pi$.

The bundle $\mc{F}_U$ on $U$ determines a section of $Gr(m_{max}, E)$ over $U$ which extends to an open set with complement of codimension at least $2$ because the Grassmannian is projective.  Denote $\mc{F}$ the smallest coherent sheaf such that restricts to this bundle over $U$.  This $\mc{F}$ satisfies the claim.
\end{proof}

We now specialize to the setting of Setup~\ref{setup_main}.  We take $V = W$ and we set $E = T_f$, a rank $m - n$ bundle.  It follows immediately from the hypothesis that $\phi^*T_W$ is globally generated.  In other words, $\phi$ is a free curve on $W$.  We may then run the entire preceding ``program'' for this map.

\begin{cor}\label{cor_glue}
The $\mc{D}(p)$ for this setup glue together in the sense that there is a subsheaf $\mc{D}$ of the relative tangent bundle $T_f$ such that over an open set $U$ of $W$, $\mc{D}_p = \mc{D}(p)$ for each $p \in U$.  The sheaf $\mc{D}$ restricts to be a vector bundle of rank $m_{max}$ on an open set whose complement has codimension at least two.
\end{cor}

Because $\mc{D}$ is a subsheaf of $T_f$ and so of $T_W$, we may restrict the Lie bracket on $T_W$ to $\mc{D}$.

\begin{lem}
The sheaf $\mc{D}$ above is integrable, that is, $[\mc{D}, \mc{D}] \subseteq \mc{D}$.
\end{lem}

\begin{proof}
Again see \cite{Mingmin_foliations}, Proposition 2.6.  Restrict $\mc{D}$ to the open set $U$ where it is a vector bundle.  Consider the following diagram,
$$
\xymatrix{
\mc{D} \otimes_{\mb{C}} \mc{D} \ar[r]^{[\_,\_]} \ar[rd] & T_f \ar[r] & T_f / \mc{D} \\
& \mc{D} \otimes_{\mc{O}_\mc{C}} \mc{D} \ar[ru]^p
}.
$$
We form $M$ as in the proof of Proposition~\ref{prop_glue_bundle}.  For a general $\phi \in M$, we have $\phi^*T_f = \oplus_{i = 1}^{m_{max}} \OO(a_i) \oplus \OO^{m - n - m_{max}}$ with each $a_i > 0 $.  Let $\pi$ be the map $M \times \PP^1 \rightarrow \mc{C}$, then we have $\pi^*\mc{D} \otimes_{\mc{O}_\mc{C}} \mc{D} = \bigoplus \OO(a_i + a_j)$ and $\pi^*(T_f / \mc{D}) = \OO^{m - n - m_{max}}$.  The map $p$ pulled back by $\pi$ corresponds to a map
$$
\bigoplus \OO(a_i + a_j) \rightarrow \mc{O}^{m - n - m_{max}}.
$$
Since each $a_i > 0$, $\pi^*p = 0$ and by a descent argument, we have $p = 0$.  This implies that $\mc{D}$ is closed under the Lie bracket, as was to be shown.
\end{proof}

\begin{lem}\label{lem-condition}
If $\phi \in M_{j_{max}}$ is general, then $\phi(\PP^1) \subseteq U$ and $\phi^*\mc{D}$ is ample.
\end{lem}

\begin{proof}
For the first statement either note that $\bigcup_{\phi \in U_{2w}} \phi(\PP^1)$ is open in $\mc{C}$ and contained in $U$, or that the free rational curves can be deformed to miss any fixed codimension two locus; see \cite{Kollar-ratcurves} Theorem II.3.7.  The last statement is clear; for a bundle to be ample on $\PP^1$ simply means it is the direct sum of ample line bundles, which it is by construction here.
\end{proof}

The existence of the distribution $\mc{D} \subseteq T_W$ is equivalent to a holomorphic foliation by the holomorphic Frobenius Theorem (see \cite{Voisin_HodgeTheoryBookI}, Section 2.3).  The following theorem of Kebekus, Sol\'{a}, and Toma \cite{KST_foliations} supplies an algebraic analogue:

\begin{thm}\label{algebraic-leaves-thm}
(See \cite{KST_foliations}) Let $V$ be a complete normal variety and $C \subseteq V$ a complete curve contained entirely in $V_{reg}$.  Suppose $F \subseteq T_V$ is a foliation which is regular and ample on $C$.  Then, for every point $x \in C$, the leaf through $x$ is algebraic.
\end{thm}

\begin{rmk}
The main result of \cite{KST_foliations} is that the leaves of the foliation are rationally connected, a fact we will not use.  At the heart of the algebraicity argument is a result due to Hartshorne related to formal neighborhoods of subvarieties.
\end{rmk}

\begin{thm}\label{thm-foliation}
The leaves of the foliation given by $\mc{D}$ on a (possibly different) open set $U \subseteq W$ are algebraic.  The complement of $U$ in $W$ still has codimension at least $2$.  Further, there is a projective variety $Y$ fitting into a commutative diagram
$$
\xymatrix{
U \ar[d]^{f|_U} \ar[dr]^h\\
X & Y. \ar[l]^\psi
}
$$
Over $U$, the relative tangent bundle $T_h$ agrees with $\mc{D}$.  The dimension of $Y$ is $m - m_{max}$.
\end{thm}

\begin{proof}
Changing notation slightly, denote by $V \subseteq W$ be the largest open set where $\mc{D}|_p = \mc{D}(p)$.  By Lemma~\ref{lem-condition}, the open set $U = \bigcup_{\phi \in U_{2w}} \phi(\PP^1)$ is contained in $V$ and for a general $\phi \in M_j$, $\phi^*\mc{D}$ is ample.  By Theorem~\ref{algebraic-leaves-thm}, on the open set $U$, every leaf of $\mc{D}$ is algebraic.  We have a map (of sets for the moment),
$$ U \rightarrow \text{Chow}(W), u \mapsto \text{ (Leaf through u) } $$
We claim that the image of this map, call it $Y$, satisfies the statement of the theorem.  Consider the incidence correspondence $\mc{I}_U \subseteq U \times W$ given by $\{(u,x)| x \in \text{ Leaf}(u)\}$.  The projection $p_2: \mc{I}_U \rightarrow U$ is proper and generically smooth, so gives rise to a map $h: U' \rightarrow \text{Chow}(W)$ where $U' \subseteq \mc{I}_U$ is the smooth locus of $p_2$; see \cite{Kollar-ratcurves} I.3.  The image of $h$ is constructible, so there is an open subvariety $Y_0$ which is dense in its closure.  Let $U'' = h^{-1}(Y_0)$.  As $\text{Chow}(W)$ is projective and we are in characteristic $0$, $Y_0$ embeds in a smooth projective variety, $Y$.  The morphism $h$ is a rational map $W \dashrightarrow Y$ which is defined on $U''$.  Two points in $U''$ have the same image in $Y$ if and only if they lie on the same leaf (the leaves of a foliation are disjoint) which implies that $\mc{D}|_{U''}$ equals the relative tangent bundle $T_h$.  Since $W$ is projective, the map extends over an open set whose complement has codimension at least $2$ in $W$.

There is at least a map of sets $\psi: Y_0 \rightarrow X$ sending all points in $\text{Leaf}(u)$ to $f(u)$.  Let $f \in \CC(X)$ be a rational function on $X$.  The pullback $ev^{-1}(f) \in \CC(W)$ is a rational function on $W$.  By construction, this function is constant on the fibers of $U'' \rightarrow Y_0$.  It then comes from a rational function $f' \in \CC(Y)$.  In other words, we have an inclusion of fields $\CC(X) \subset \CC(Y)$ and so a dominant rational map $\psi: Y \dashrightarrow X$.  By replacing $Y$ with a blowup, we may as well assume this map is everywhere defined.
\end{proof}

We consider the picture restricted over a general point $x \in X$.
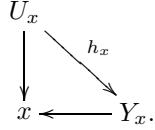
$$
\xymatrix{
U_x \ar[d] \ar[dr]^{h_x}\\
x & Y_x. \ar[l]
}
$$

\begin{prop}\label{prop_contract_curve_ample}
Suppose that the map $h_x$ contracts a complete curve contained in $U_x$.  Then $m_{max} = m - n$, and so for a general $[\phi] \in M_{j_{max}}$, $\phi$ is relatively very free.
\end{prop}
\begin{proof}
Define $Z_x$ to be the complement of $U_x$ in $W_x$.  Note that the codimension of $Z_x$ in $W_x$ is at least $2$.  We claim that $H^2(W_x, \ZZ) \rightarrow H^2(U_x, \ZZ)$ is an isomorphism.  Similarly $H_2(U_x, \ZZ) \rightarrow H_2(W_x, \ZZ)$ is an isomorphism.  To verify the claim, set $Z_x^1 = Z_x - (Z_x)_{sing}$, $W_x^1 = W_x - (Z_x)_{sing}$.  Note that $U_x = W_x^1 - Z_x^1$.  We have the Gysin sequence in cohomology whose relevant portion looks like:
$$
H^2_{Z_x^1}(W_x^1) \rightarrow H^2(W_x^1) \rightarrow H^2(U_x) \rightarrow H^3_{Z_x^1}(W_x^1)
$$
(all coefficients are understood to be $\ZZ$).  As the codimension of $Z_x^1$ in $X$ is $c \geq 2$ we have that $H^2_{Z_x^1}(W_x^1) = H^{2 - 2c}(Z_x^1) = 0$ and $H^3_{Z_x^1}(W_x^1) = H^{3 - 2c}(Z_x^1) = 0$ so that $H^2(W_x^1) \cong H^2(U_x)$.  Now set $Y^2 = (Z_x)_{sing}$ and let $Z_x^2 = Y^2 - Y^2_{sing}$ and $W_x^2 = W_x - Y^2_{sing}$.  We repeat the same argument above to get that $H^2(W_x^2) \cong H^2(W_x^1)$.  Repeating the argument implies that $H^2(W_x) \cong H^2(U_x)$.  Thus, $H^2(U_x)$ has rank $1$.  By the homology exact sequence of a pair and the Thom isomorphism in homology, the same argument shows that $H_2(U_x) \cong H_2(W_x)$.

If $h_x$ contracts a complete curve contained in $U_x$, then because this curve represents a homology class in $U_x$, the map $h_x$ must map all of $U_x$ to a point.  This implies then that the rank of $T_h$ is maximal, namely $m - n$.  The claim then follows by construction of the foliation, Lemma~\ref{lem_max_dim}, Corollary~\ref{cor_glue}, and Theorem~\ref{thm-foliation}.
\end{proof}

\section{The Geometric Construction}\label{section_geometry}

Using the foliation in Theorem~\ref{thm-foliation} and the contraction result of Proposition~\ref{prop_contract_curve_ample}, the goal is to prove the existence of complete curves in $U_x$ contracted by $h_x$.  In the setting of Setup~\ref{setup_main}, we will show that a large dimensional family of (possibly affine) curves is contracted by the map $h_x$.

Let $M_{1,2}$ be the unique component of $\kspace{2}{W}{\beta}$ dominating $M_1$.  There are two evaluation maps $ev_1, ev_2: M_{1,2} \rightarrow W$.  Points of $M_{1,2}$ will be denoted $(\phi, p, q)$.  For a general $(\phi, p, q) \in M_{1,2}$, the composition map $f \circ \phi: \PP^1 \rightarrow W \rightarrow X$ is a smooth point of $\kspace{0}{X}{\alpha}$, by Hypothesis A.  The point $(f \circ \phi)$ then, is contained in a unique irreducible component, call it $M_X$, of the mapping space $\kspace{0}{X}{\alpha}$.  Let $M_{X,2}$ be the unique component of $\kspace{2}{X}{\alpha}$ dominating $M_{X}$.  For $i = 1,2$ as above, there are evaluation maps $ev_i: M_{X,2} \rightarrow X$.  Form the fiber product $F = M_{X,2} \times_{ev_1, X, f} W$.  A point of $F$ corresponds to the data of a stable, genus $0$ map $C \rightarrow X$, two points $r,s \in C$, and a point $w \in W$ such that $r$ maps to $f(w)$.

There is a map from $M_{1,2} \rightarrow F$ which sends a point $(\phi, p, q)$ to $((f \circ \phi, p, q), \phi(p))$.  Let $\mc{M}$ be the fiber product $M_{1,2} \times_F M_{1,2}$.  A point of $\mc{M}$ may be identified with two, two-pointed maps, $((\phi, p, q), (\phi', p', q'))$ satisfying the following conditions.

\begin{cond}\label{conditions}
Both $(\phi, p, q)$ and $(\phi', p', q')$ are points of $M_{1,2}$ such that
\begin{enumerate}
\item The two maps $f \circ \phi$ and $f \circ \phi'$ determine the same point of $M_X$ up to stabilization (that is, via the map $F: M_{1,2} \rightarrow M_{X,2})$
\item $\phi(p) = \phi'(p')$
\item $f(\phi(q)) = f(\phi'(q'))$.
\end{enumerate}
\end{cond}

Consider the composition of the $pr_2: \mc{M} \rightarrow M_{1,2}$ and $ev_2: M_{1,2} \rightarrow W$.  $$\Psi: \mc{M} \rightarrow W$$ $$((\phi, p, q), (\phi', p', q')) \mapsto \phi'(q').$$

We only wish to consider points of $((\phi, p, q), (\phi', p', q'))$ of $\mc{M}$ which are close to $(\phi, p, q)$ in some sense.  The issue is that the fiber of $\Psi$ (and of $pr_2$) may not be irreducible.  To get around this problem, we use the fact that there is a section of the map $pr_1$, namely the diagonal map $\Delta$.  Given a general $(\phi, p , q) \in M_{1,2}$, we define $D_{\phi, p, q}$ to be the component of $pr_1^{-1}(\phi, p, q)$ which contains $\Delta(\phi,p,q)$.  This is possible as the image of $\Delta$ is generically smooth. Then define $M_{\phi,p,q} = pr_2(D_{\phi, p, q})$ and $C_{\phi,p,q} = \Psi(D_{\phi,p,q})$.

We now formulate a key hypothesis on the smoothness of the whole setup which will be sufficient to prove the main theorem.  Let $R$ be the non-empty, open locus in $M_{1,2}$ such that for each $(\phi, p, q) \in R$ the sheaf $\phi^*T_f$ is globally generated.

\begin{itemize}
\item (Hypothesis C) Let $w \in W$ be a general point, and let $(\phi, p, q) \in R$ be general such that $\phi(p) = w$.  Let $z$ be any point lying over $f(\phi(q))$, and let $(\phi', p', q') \in M_{\phi,p,q}$ be \textit{any} point satisfying the conditions in \ref{conditions} and $\phi'(q') = z$.  The hypothesis is that such a $\phi'$ corresponds to an unobstructed point of the map $ev_2$.
\end{itemize}

\begin{rmk}\label{rmk_weaken_hypothesisC}
It is automatic that such $\phi'$ are smooth points when $\phi'$ corresponds to a map from a smooth rational curve.  This is because $\phi'$ passes through the point $\phi(p)$, which is a general point of $W$.  That the map $\phi'$ is then unobstructed follows from Theorem II.3.11 of \cite{Kollar-ratcurves}.  Thus, Hypothesis C may be replaced with the assumption that the corresponding fact is true at every point of the boundary of $M_{1,2}$.  This will ease the work in verifying the Hypothesis in applications if the curve class $\beta$ is ``close" to being indecomposable.
\end{rmk}

The following dimension computations and estimates are recorded in the following.

\begin{lem}\label{lem_dim_count}
\begin{enumerate}[(i)]
\item The dimension of $\mc{M}$ is $2(-K_W \cdot \beta) + K_X \cdot \alpha + m - 1$.
\item For a general point $w \in W$, the dimension of $\Psi^{-1}(w)$ is $2(-K_W \cdot \beta) + K_X \cdot \alpha - 1$. Every fiber of $\Psi$ has at least this dimension.
\item The dimension of a general fiber of $pr_i : \mc{M} \rightarrow M_{1,2}$ is $-K_W \cdot \beta + K_X \cdot \alpha$.  Every fiber of $pr_i$ has at least this dimension.
\item The dimension of a general fiber of $ev_2: M_{1,2} \rightarrow W$ is $-K_W \cdot \beta - 1$.  Every fiber has at least this dimension.
\item If there is a codimension $2$ locus $Z \subset W$ such that $\Psi^{-1}(Z)$ dominates $M_{1,2}$ via the restriction of the map $pr_1$, then there is a point $z \in Z$ such that $\Psi^{-1}(z)$ has dimension at least $2(-K_W \cdot \beta) + K_X \cdot \alpha$ (one more than expected).  Further, we may find such a point $z$ such that this inequality holds and such that $pr_1(\Psi^{-1}(z))$ contains a general point of $M_{1,2}$.
\end{enumerate}
\end{lem}
\begin{proof}
As the stacks being discussed are generically smooth, they have the expected dimension which may be calculated in the standard way.  Items (i)-(iv) are a straightforward application of this; the last part of (v) follows because, by assumption, $\Psi^{-1}(Z)$ maps dominantly to $M_{1,2}$.
\end{proof}

\begin{lem}\label{lem_describe_defs}
For a general $(\phi, p, q) \in M_{1,2}$, the variety $M_{\phi,p,q}$ has dimension $h^0(\PP^1, \phi^*T_f(-1))$ and is smooth at $(\phi, p, q)$.  The locus $M_{\phi,p,q} \cap \kspacenb{2}{W}{\beta}$ consists of smooth points of $M_{1,2}$.
\end{lem}
\begin{proof}
We claim the tangent space to $M_{\phi,p,q}$ at $(\phi, p, q)$ is $H^0(\PP^1, \phi^*T_f(-p))$.  This follows from Lemma~\ref{lem_relative_def_fix_pt} as a tangent vector to $M_{\phi,p,q}$ at $(\phi, p, q)$ corresponds to an infinitesimal deformation of $\phi$ fixing $f \circ \phi$ and $\phi(p)$.

Since $[\phi] \in M_1$, the vector space $H^1(\PP^1, \phi^*T_f(-p))$ is zero by Hypothesis B.  This implies that $M_{\phi,p,q}$ is smooth at the point $(\phi, p, q)$ and the first claim follows.  To check the last statement, suppose $(\phi', p', q') \in M_{\phi, p, q} \cap \kspacenb{2}{W}{\beta}$.  The map $\phi'$ corresponds to a map from an irreducible curve and $\phi'(q_1') = q$ is a general point of $W$.  By Theorem II.3.11 of \cite{Kollar-ratcurves}, this map is free and the claim follows.
\end{proof}


Set $x = f(\phi(q))$.  The image of $M_{\phi, p, q}$ in $W$ is contracted by $h_x$ in the following sense.

\begin{lem}\label{lem_contraction}
The variety $C_{\phi, p, q}$ is positive dimensional.  Further, $h_x(C_{\phi,p, q} \cap U_x)$ is a point.
\end{lem}
\begin{proof}
If the image of $M_{\phi, p, q}$ were a point in $W_x$, then each point of $M_{\phi, p, q}$ would pass through $\phi(q)$.  This is impossible though as $\dim H^0(\PP^1, \phi^*T_f(-p - q)) < \dim M_{\phi, p , q}$.  In other words, the subvariety parameterizing maps $(\phi', p', q')$ with $\phi(q) = \phi(q')$ must have dimension smaller that $M_{\phi, p, q}$ by considering the dimension of the tangent spaces.

The second claim will follow if we show that the tangent space to $C_{\phi,p, q}$ at the image of the point $(\phi,p,q)$ is contained in $\mc{D}$.  A tangent vector to $C_{\phi,p,q}$ at $\phi(q)$ is identified with an element of the vector space $H^0(\PP^1, \phi^*T_f(-p))$ inside $T_{W,{\phi(q)}}$.  The bundle $\phi^*T_f(-p)$ splits into a direct sum of line bundles, and we identify this with $\mc{E} \oplus \OO(-1)^a$.  In other words the tangent vector, considered as a global section of $\phi^*T_f|_q = (\mc{E} \oplus \OO^a)|_q$ is contained in the positive piece, $\mc{E}$.  By definition, this element is in $\mc{D}$.
\end{proof}

The Lemma shows that there is a supply of (possibly affine) curves in the fibers of $h_x$.  Namely any curve on a $C_{\phi, p, q}$ which meets $U_x$. We wish to show that one of these curves in $C_{\phi,q, x}$ (and so the general one), is actually contained in $U_x$ (see Proposition~\ref{prop_contract_curve_ample}).  Given these preliminaries, we can prove the main result.

\begin{thm}\label{thm_rel_vf}
Suppose we are in the situation of Setup~\ref{setup_main} and that, in addition, Hypothesis C holds.  Then there is a map $\phi': \PP^1 \rightarrow W$ such that $\phi'$ is relatively very free.
\end{thm}

\begin{proof}
By Proposition~\ref{prop_contract_curve_ample}, the Theorem follows if there is a complete curve in the fiber of $h_x: U_x \rightarrow Y_x$ for a general $x \in X$.  Assume, by way of contradiction, that this is not the case.  Then no complete curve on $C_{\phi,p,q}$ can be contained in $U_{f(\phi(q))}$ by Lemma~\ref{lem_contraction}.  This implies that $\Psi^{-1}(Z)$ maps dominantly to $M_{1,2}$ via $pr_1$.  Otherwise, for a general point $(\phi, p, q) \in M_{1,2}$, any complete curve on $\Psi(pr_1^{-1}(\phi, p , q)) = C_{\phi, p , q}$ would miss $Z$.

By Lemma~\ref{lem_dim_count} then, there is a point $z \in Z$ such that $\Psi^{-1}(z)$ has dimension at least $2(-K_W \cdot \beta) + K_X \cdot \alpha$ (one more than expected) and such that $pr_1(\Psi^{-1}(z))$ contains a general point of $M_{1,2}$.  Taking an open set $H$ in $\Psi^{-1}(z)$, we may assume that each point of $pr_1(H)$ is general in the sense that for each $\phi \in pr_1(H)$ the sheaf $\phi^*T_f$ is globally generated.

Consider the restriction of the second projection map, $pr_2|_H: H \rightarrow M_{1,2}$.  Let $\textbf{p} = (\phi, p, q, \phi', p', q')$ be a point of $H$.  We claim that $pr_2|_H$ is smooth at $\textbf{p}$.  Let $s = pr_2(\textbf{p})$.  Note that the image of $pr_2|_H$ may be identified with an open set of $s$ inside the fiber of $ev_2: M_{1,2} \rightarrow W$ over $z$.  Let $S \subset ev_2^{-1}(z)$ be this image.  By Hypothesis $C$, $S$ is smooth.

The fiber of $pr_2|_H$ over $s = (\phi',p',q')$ may be identified with a $\PP^1$ bundle over the points $(\phi, p)$ of $M_{1,1}$ such that the image of $\phi$ and of $\phi'$ agree in $X$ and such that $\phi(p) = \phi'(p')$.  Using this description, a tangent vector to the fiber of $pr_2|_H$ at $\textbf{p}$ is identified with an element of $H^0(\PP^1, \phi^*T_W)/H^0(\PP^1, T_\PP^1(p + q))$ such that the image in $H^0(\PP^1, \phi^*f^*T_X)/H^0(\PP^1, T_\PP^1(p + q))$ is zero and the image in $H^0(\PP^1, \phi^*T_W|_p)$ is zero.  In other words, the tangent vectors to the fiber may be identified with elements of $H^0(\PP^1, \phi^*T_f(-p))$.  This deformation problem may be formulated as in Lemma~\ref{lem_relative_def_fix_pt}.  The obstruction to lifting infinitesimal deformations of this problem are contained in the vector space $H^1(\PP^1, T_f(-p))$.  By construction, $(\phi, p, q)$ is a general point of $M_{1,2}$ in the sense that $\phi^*T_f$ is globally generated and so $H^1(\PP^1, T_f(-p)) = 0$.  This implies that the fiber is smooth at $\textbf{p}$.

The conclusions of the above argument may be rephrased by saying that there is a short exact sequence $$0 \rightarrow K \rightarrow T_{\Psi^{-1}(z)} \rightarrow pr_2^*T_S \rightarrow 0$$

which, when restricted to the point $\textbf{p}$, gives the short exact sequence
$$0 \rightarrow H^0(\PP^1, \phi^*T_f(-1)) \rightarrow T_{\Psi^{-1}(z), \textbf{p}} \rightarrow H^0(\PP^1, \phi^*T_W(-q))/H^0(\PP^1, T_{\PP^1}(p + q)) \rightarrow 0.$$

This implies that the tangent space to $T_{\Psi^{-1}(z)}$ at $\textbf{p}$ has dimension $h^0(\PP^1, \phi^*T_f(-1)) + h^0(\PP^1, \phi^*T_W(-1)) - 1 = 2(-K_W \cdot \beta) + K_X \cdot \alpha - 1$.  Thus there can be no point $z \in Z$ such that $\dim(\Psi^{-1}(z)) = 2(-K_W \cdot \beta) + K_X \cdot \alpha$.  This contradiction finishes the proof of the Theorem.
\end{proof}

\section{Preliminaries on Twisting Surfaces}\label{section_twisting}

Let $X \subset \PP^n$ be a smooth projective variety.  Let $F(X) = \kspace{0}{X}{1}$ be the Fano scheme of lines on $X$ (embedded as a subscheme of the Grassmannian).  Let $\mc{C}$ be the universal line on $X$.

Let $\Sigma = \PP^1 \times \PP^1$ and let $\pi: \Sigma \rightarrow \PP^1$ be the first projection map.  Denote by $F$ the class of a fiber and $F'$ the class of a square zero section on $\Sigma$.  If $f: \Sigma \rightarrow X$ is a morphism, we may consider the associated map $(\pi, f): \Sigma \rightarrow \PP^1 \times X$.  The normal sheaf of this map will be denoted $\mc{N}_f$.  The following definition is taken from \cite{dJS}, Section 7.

\begin{defn}
Suppose $\Sigma = \PP^1 \times \PP^1$ and $\pi: \Sigma \rightarrow \PP^1$ is the first projection map.  For an integer $m > 0$ a map $f: \Sigma \rightarrow X$ is an \textit{$m$-twisting surface} on $X$ if
\begin{enumerate}
\item The sheaf $f^* T_X$ is globally generated.
\item The map $(\pi,f)$ is finite and $H^1(\Sigma, \mc{N}_f(-F' - mF)) = 0$.
\end{enumerate}
\end{defn}

We will study this condition in order to rephrase it in terms of maps $\PP^1$ to the universal line $\mc{C}$ on $X$.

A map $f: \PP^1 \rightarrow F(X)$ determines and is determined by a diagram
\begin{equation}\label{diag_surf}
\xymatrix{
\Sigma \ar[d]^{\pi'} \ar[r]^{f'} \ar@/^1.5pc/[rr]^h & \mc{C} \ar[r]^{ev} \ar[d]^{\pi} & X \\
\PP^1 \ar[r]^f & F(X) }
\end{equation}
where $\Sigma \cong \PP^1 \times_{F(X)} \mc{C}$ is a $\PP^1$-bundle over $\PP^1$, $h$ is the composition of the two horizontal maps, and $h$ maps fibers of $\pi'$ to lines on $X$.

Analogously, the data of a map $f: \PP^1 \rightarrow \mc{C}$ is equivalent to a similar diagram, but with a section, $\sigma$, of the map $\pi'$:
$$
\xymatrix{
\Sigma \ar[d]^{\pi'} \ar[r]^h & X \\
\PP^1 \ar@/^1pc/[u]^{\sigma} &.
}
$$
Here, each fiber of $\pi'$ is a $\PP^1$ mapped to a line on $X$, and $\sigma$ is a section of $\pi'$.  From now on we replace $\pi'$ with $\pi$ and trust no confusion will result.

\begin{defn}
A map $f:\PP^1 \rightarrow \mc{C}$ as above is a \textit{family of pointed lines}.
\end{defn}

The discussion above implies that a family of pointed lines $f: \PP^1 \rightarrow \mc{C}$ is equivalent to the data $(\Sigma, \pi, \sigma, h)$.

\begin{rmk}
The obstruction space to a map $g: \PP^1 \rightarrow X$ is given by $\mathbb{E}xt^2(g^*\Omega_X \rightarrow \Omega_{\PP^1}, \OO_{\PP^1})$.  Here $g^*\Omega_X \rightarrow \Omega_{\PP^1}$ is the relative cotangent complex for $g$, denoted $L_g$.  Using the long exact sequence for hyperext and the fact that $Ext^2(\Omega_{\PP^1}, \OO_{\PP^1}) = Ext^1(\Omega_{\PP^1}, \OO_{\PP^1}) = 0$, we conclude that
$$\mathbb{E}xt^2(L_g, \OO_{\PP^1}) \cong Ext^1(g^*\Omega_X, \OO_{\PP^1}) \cong H^1(\PP^1, g^*T_X).$$

Using this description, we see that for a map $f: \PP^1 \rightarrow \mc{C}$, the condition that $\mathbb{E}xt^2(L_{ev \circ f}, \OO_{\PP^1}) = 0$ is equivalent to the condition that $(ev \circ f)^* T_X$ splits as $\bigoplus \OO(a_i)$ with each $a_i \geq -1$ (in other words, that the section of the associated $\pi: \Sigma \rightarrow \PP^1$ is mapped to an unobstructed curve on $X$).  Similarly, if $\mathbb{E}xt^2(L_{f_t}, \OO_{\PP^1}) = 0$ for each $t \in \PP^1$, then each fiber of $\pi$ corresponds to a smooth point of $F(X)$.
\end{rmk}

We are actually interested in a stronger positivity condition.

\begin{defn}
A family of pointed lines $f: \PP^1 \rightarrow \mc{C}$ will be called
\begin{enumerate}
\item \textit{Section free} if $(ev \circ f)^* T_X$ splits as $\bigoplus \OO(a_i)$ with each $a_i \geq 0$.
\item \textit{Fiberwise free} if $h_t: \Sigma_t \rightarrow X$ is free for each $t \in \PP^1$.
\end{enumerate}
\end{defn}

Observe that if $f:\PP^1 \rightarrow \mc{C}$ is a family of pointed lines and the map $(\pi,h)$ is an embedding, then sheaf $\mc{N}_f$ is a vector bundle.  In this case, the sheaf $\mc{N}_f$ is $\pi$-flat; see \cite{dJS} Lemma 7.1.

\begin{lem}
If $f:\PP^1 \rightarrow \mc{C}$ is a family of pointed lines which is fiberwise free, then $R^1\pi_*\mc{N}_f(-\sigma) = 0$.
\end{lem}
\begin{proof}
As $\pi$ is a submersion, $$\mc{N}_f|_{\Sigma_t} \cong h_t^*T_X/T_{\Sigma_t} \cong N_{\Sigma_t/X}$$ for each $t \in \PP^1$.  Since $h_t(\Sigma_t)$ is a free line on $X$, $H^1(\Sigma_t, \mc{N}_f(-\sigma)|_{\Sigma_t}) = 0$.  This implies the claim.
\end{proof}

\begin{notat}
Denote by $\mc{C}_{ev}$ the inverse image under $\pi: \mc{C} \rightarrow F(X)$ of all points in $F(X)$ corresponding to free lines.  Let $T_{ev}$ be the dual of the sheaf of relative differentials for the map $ev|_{\mc{C}_{ev}}$.
\end{notat}

\begin{prop}\label{rel-ev}
If $f: \PP^1 \rightarrow \mc{C}$ is a family of pointed lines which is fiberwise free, then $f^*T_{ev} = \pi_*(\mc{N}_f(-\sigma))$.
\end{prop}

\begin{proof}
We will continue to denote the map $\Sigma \rightarrow X$ by $h$.

For any $[p \in l] \in Im(f) \subseteq \mc{C}$ we have that
$$
T_{\mc{C}}|_{[p,l]} = \text{Coker } \left( H^0(l, T_l(-p)) \rightarrow H^0(l, ev^*T_X|_l) \right).
$$
This follows from the corresponding sequence for hyper-Ext discussed above and the fact that $H^1(l, T_l(-p)) = 0$.  Thus,
$$
f^*T_{\mc{C}} = \text{Coker } \left(\pi_*T_{\pi}(-\sigma) \rightarrow \pi_*h^*T_X\right).
$$
Twisting down the sequence
$$
0 \rightarrow T_{\pi} \rightarrow T_{\Sigma} \rightarrow \pi^*T_{\PP^1} \rightarrow 0
$$
by $\sigma$ and pushing it forward by $\pi$, implies that
$$
\pi_*T_{\pi}(-\sigma) \cong \pi_* T_{\Sigma}(-\sigma)
$$
because $\pi_* \pi^* (T_{\PP^1}(-\sigma)) = 0$.  Using the exact sequence
$$
0 \rightarrow T_{\pi} \rightarrow h^*T_X \rightarrow \mc{N}_f \rightarrow 0
$$
and computing that $R^1\pi_*(T_{\pi}(-\sigma)) = 0$, we conclude that
$$
0 \rightarrow \pi_*T_{\pi}(-\sigma) \rightarrow \pi_*h^*T_X(-\sigma) \rightarrow \pi_* \mc{N}_f(-\sigma) \rightarrow 0
$$
is a short exact sequence.  Putting this all together and using $R^1\pi_*h^*T_X(\sigma) = 0$ gives
$$
\xymatrix{
& & 0 & 0 \\
& & h^*T_X|_{\sigma} \ar[u] \ar@{=}[r] & f^*ev^*T_X \ar[u]\\
0 \ar[r] & \pi_*(T_{\pi}(-\sigma)) \ar[r] \ar@{=}[d] & \pi_*h^*T_X \ar[u] \ar[r] & f^*T_{\mc{C}} \ar[r] \ar[u] & 0 \\
0 \ar[r] & \pi_*(T_{\pi}(-\sigma)) \ar[r] & \pi_*h^*T_X(-\sigma) \ar[u] \ar[r] & \pi_*\mc{N}_f(-\sigma) \ar[u] \ar[r] & 0\\
& & 0 \ar[u] & 0 \ar[u] &.}
$$
Since there is also the short exact sequence
$$
0 \rightarrow f^*T_{ev} \rightarrow f^*T_{\mc{C}} \rightarrow f^*ev^*T_X \rightarrow 0,
$$
we conclude that $f^*T_{ev} \cong \pi_*\mc{N}_f(-\sigma)$.
\end{proof}

We are now able to rephrase the cohomological properties satisfied by an $m$-twisting surface.

\begin{cor}
Let $\Sigma = \PP^1 \times \PP^1$ be a ruled surface and $h: \Sigma \rightarrow X$ a map so that $\Sigma \rightarrow \PP^1 \times X$ is an embedding such that the fibers of $\pi: \Sigma \rightarrow \PP^1$ are mapped to free lines on $X$.  Let $\sigma$ be a square zero section and $f: \PP^1 \rightarrow \mc{C}$ the induced map.  Then $f$ is an $m$-twisting surface on $X$ if $H^1(\PP^1, f^*T_{ev}(-m)) = 0$.  Conversely, an $m$-twisting surface whose fibers map to lines on $X$ corresponds to a family of lines which is fiberwise free and with $H^1(\PP^1, f^*T_{ev}(-m)) = 0$.
\end{cor}

\begin{proof}
The first statement follows from the Proposition and an application of the Leray spectral sequence.  The second statement follows from the fact that free curves on $\Sigma$ map to free curves on $X$; see \cite{dJS} Lemma 7.4.  Then such a $\Sigma$ corresponds to a family of free lines on $X$ and the Proposition may be applied.
\end{proof}

\section{Application: Complete Intersections}\label{section_ci}

\subsection{Lines and 2-Planes on Complete Intersections}

Let $X \subset \PP^n$ be a smooth complete intersection of type $\underline{d} = (d_1, \ldots, d_c)$ with each $d_i \geq 2$.  Suppose that $\dim(X) \geq 3$ so that $Pic(X) = \ZZ$.  We use $F(X)$ and $\mc{C}$ to respectively denote the schemes of lines and pointed lines on $X$.   These are projective varieties which we always considered to be embedded by the Pl\"ucker map.  There is a diagram
$$
\xymatrix{
\mc{C} \ar[r]^{ev} \ar[d]^{\pi} & X \\
F(X).
}
$$

Set $\textbf{d} = \Sigma_{i=1}^c d_i$.  For $x \in X$, denote by $F(X, x)$ the lines in $\PP^n$ which contain $x$ and are contained in $X$.  Equivalently, $F(X, x) = ev^{-1}(x)$.  A linearly embedded $\PP^2 \subset \PP^n$ will be referred to as a $2$-plane.  We record the following standard fact.

\begin{lem}\label{lem_2plane}
Given $X$ a smooth complete intersection as above.  A line on $F(X)$ may be identified with a $2$-plane $P$ contained in $X$ and a point $x \in P$.  For a general line $l \subset X$ and a point $x \in l$, there is an isomorphism between the space of lines on $F(X, x)$ through the point $[l]$ and the space of $2$-planes contained in $X$ and containing $l$.
\end{lem}

We record the following facts about lines and $2$-planes on $X$.  Many of these are contained in the paper by Debarre and Manivel, see \cite{DebarreManivel_Linear}.

\begin{thm}\label{thm_numerical_lines}
\begin{enumerate}[(i)]
\item If $2n - \textbf{d} - 3 \geq 0$, then $F(X)$ is non-empty.
\item If $2n - \textbf{d} - 3 \geq 0$, the scheme $F(X)$ is smooth of dimension $2n - \textbf{d} - 3$ for a general complete intersection $X$.
\item If $F(X)$ is smooth and if $n \geq 2\textbf{d} + 4$ then $Pic(F(X))$ is isomorphic to $\ZZ$, generated by $\OO(1)$.
\item The scheme $F(X,x)$ may be identified with a projective scheme defined by equations of degrees $(1, 2, \ldots, d_1, 1, 2, \ldots, d_2, \ldots, 1, 2, \ldots, d_c)$ inside $\PP^{n-1}$.  For a general $x \in X$ the scheme $F(X, x)$ is a smooth complete intersection of this type.  If $3 + c \leq n$, then for a general $x \in X$, $Pic(F(X,x)) = \ZZ$, generated by $\OO(1)$.
\item Define $\gamma(n, d_i) = n - 2 - \frac{1}{2} \Sigma(d_i^2 + d_i)$.  If $\gamma > 0$ then for a general line $l \subset X$, the family of $2$-planes in $\PP^n$ contained in $X$ and containing $l$ is smooth, irreducible, and has dimension $\gamma(n, d_i)$.
\end{enumerate}
\end{thm}

\begin{proof}
For (i) - (iii), see Th\'eor\`eme 2.1(b) and Proposition 3.1 of \cite{DebarreManivel_Linear}.  For part (iv), to see that the space of lines through a fixed point $x \in X$ is cut out by equations of this degree is an easy computation in coordinates.  For a general point $x \in X$, every line through $x$ is free, and so the variety $F(X,x)$ is smooth and must be a complete intersection.  The statement about the Picard group follows from the Lefschetz Theorem.  Finally for (v), we may simply apply part (iv) to the variety $F(X,x)$ in light of the Lemma~\ref{lem_2plane}
\end{proof}

\subsection{Existence of 1-Twisting Surfaces}

\begin{sit}\label{sit_ci}
 We assume that $X \subset \PP^n$ is a smooth complete intersection of type $\underline{d} = (d_1, \ldots, d_c)$ with $\dim(X) \geq 3$ and each $d_i \geq 2$.  In this section we assume that $n \geq \Sigma d_i^2$.
 \end{sit}

Here we cite the results in \cite{dJS} which imply the existence of $1$-twisting surfaces on $X$.  We remark that while their method is certainly ingenious, it is possible that a more ``direct" proof exists.  One can show that on any complete intersection $X$ as in Situation~\ref{sit_ci}, a general line $l \subset X$ is contained in a smooth quadric surface $\Sigma \subset X$.  If one can show that a general smooth conic (or even a general pair of intersecting lines) on $X$ is contained in a smooth quadric surface then this would be enough to conclude that such surfaces are $1$-twisting.  The methods of \cite{dJS} imply this a posteriori by an indirect method which involves exhibiting $X$ as a hyperplane section of larger dimensional smooth complete intersection.

\begin{thm}\label{thm_one_twist_ci}(\cite{dJS})
Given a smooth $X$ as in Situation~\ref{sit_ci}, then $X$ contains a smooth quadric surface $\Sigma$ which is $1$-twisting.
\end{thm}
\begin{proof}
This follows from \cite{dJS} Corollary 6.11 (take $m = 2$) and Lemmas 7.8 and 7.10.
\end{proof}

\begin{rmk}
This theorem may be shown by carrying out the program in the preceding paragraph directly in the case of smooth cubic hypersurfaces (see \cite{D2}) in $\PP^9$, smooth quadric hypersurfaces (trivially) in $\PP^4$, and smooth $(2,2)$ complete intersections in $\PP^8$ via synthetic arguments.
\end{rmk}

\subsection{Existence of 2-Twisting Surfaces}

By Theorem~\ref{thm_one_twist_ci}, there is a smooth quadric surface embedded in $X$ by $h: \Sigma \rightarrow X$ such that the corresponding cohomology group $H^1(\Sigma, \mc{N}( -\sigma - F)) = 0$ where $\sigma$ is a square zero section class and $F$ is the class of a fiber for the first projection map $\pi: \Sigma \rightarrow \PP^1$.  Using the Leray spectral sequence, we conclude that $H^1(\PP^1, \pi_*\mc{N}( -\sigma - F)) = 0$ as well.

Such a quadric surface on $X$ with a square zero section corresponds to a map $f: \PP^1 \rightarrow \mc{C}$.  Each line on a smooth quadric surface is a free line, and so by Lemma 7.4 of \cite{dJS}, these fibers also map to free lines on $X$.  In other words, the map $f$ is fiberwise free, and so by Proposition~\ref{rel-ev} we have $f^*T_{ev} = \pi_*(\mc{N}_f(-\sigma))$.  Now $\pi_*(\mc{N}_f(-\sigma - F)) = \pi_*(\mc{N}_f(-\sigma))(-1)$ and so that $f^*T_{ev}$ is a globally generated vector bundle on $\PP^1$.  Likewise, the section curve $\sigma \subset \Sigma$ is mapped to a free line on $X$ and so $\sigma^*h^*T_X = f^*ev^*T_X$ is also globally generated on $\PP^1$.

\begin{lem}\label{lem_grr}
If $X$ is a smooth complete intersection as in Situation~\ref{sit_ci} and if a $1$-twisting quadric surface $h: \Sigma \rightarrow X$ corresponds to the map $f:\PP^1 \rightarrow \mc{C}$ as discussed above, then $\deg f^*T_{ev} = n + 1 - \sum d_i^2$.
\end{lem}
\begin{proof}
This follows by applying the Grothendieck Riemann Roch formula and using the identification $f^*T_{ev} = \pi_*(\mc{N}_f(-\sigma))$.
\end{proof}

\begin{sit}\label{sit_ci2}
Let $X$ be a smooth complete intersection as in Situation~\ref{sit_ci}.  We additionally assume that $n = \sum d_i^2$ and that the space of lines, $F(X)$, is smooth.
\end{sit}

\begin{prop}\label{prop_hypAB_ci}
Let $X$ be as in Situation~\ref{sit_ci2}.  Setting $\mc{C} = W$ and $ev = f$, the hypotheses A and B from Setup~\ref{setup_main} are satisfied.
\end{prop}

\begin{proof}
As $F(X)$ is assumed to be smooth, so also is $\mc{C}$.  By Theorem~\ref{thm_one_twist_ci}, there is a quadric surface on $X$ which is $1$-twisting.  This corresponds to a map $f: \PP^1 \rightarrow \mc{C}$, and by the discussion above, we have that $f^*T_{ev}$ and $f^*ev^*T_X$ are globally generated.  By Lemma~\ref{lem_grr}, the bundle $f^*T_{ev}$ has degree $1$. For a general $x \in X$, the fiber $ev^{-1}(x)$ is a smooth complete intersection with Picard group $\ZZ$ by Theorem~\ref{thm_numerical_lines}(iv) and the dimension assumptions.  This finishes the claim.
\end{proof}

It remains to verify Hypothesis C from Section 4.

Let $f: \PP^1 \rightarrow \mc{C}$ be a general map corresponding to a $1$ twisting surface on $X$.  Let $\beta$ denote the class of $f_*[\PP^1]$ in $H^2(\mc{C}, \ZZ)$.  Let $p$ and $q$ be general points on this $\PP^1$.  The map $f$ corresponds to a smooth quadric surface on $X$ with two distinguished pointed lines (corresponding to the images of $p$ and $q$), call them $(x_1, l_1)$ and $(x_2, l_2)$.  Each of these lines is a fiber of the map $\pi: \Sigma \rightarrow \PP^1$ and each is pointed at its intersection with the distinguished section (also a line on $X$).  The following generic properties are upshots of having the positivity of $1$-twisting surfaces.

\begin{lem}\label{lem_generic_properties}
We may assume that the map $f$ satisfies the following ``generic" properties.
\begin{itemize}
\item Every line through $x_1$ and $x_2$ is free.
\item The lines $\sigma$ and $l_1$ on the corresponding $\Sigma$ may be assumed to be a general pair of intersecting lines through $x_1$.
\end{itemize}
\end{lem}

\begin{proof}
The first property holds as these surfaces are constructed by taking rational curves in the fibers of the map $\kspace{2}{X}{2} \rightarrow X^2$ which is shown to be dominant; see \cite{dJS} Corollary 5.10, Corollary 6.11, and Lemma 7.8.

The second property holds by picking any one twisting surface $\Sigma$ on $X$.  Let $D$ be the nodal curve on $\Sigma$ consisting of a square zero section and any fiber.  A general deformation of $D$ in $X$ (as a nodal curve) is contained in a deformation of $\Sigma$ in $X$.  The deformation of $\Sigma$ remains $1$-twisting.  See Lemma 7.4 of \cite{dJS}.
\end{proof}

 To verify Hypothesis C, we must verify that any stable, genus $0$ map $g: B \rightarrow \mc{C}$ with class $\beta$ passing through $f(p)$ and such that $ev(f(\PP^1)) = ev(g(B))$ is unobstructed as a curve on $\mc{C}$ passing through the point lying over $f(q)$.  The data of such a $g$ is equivalent to a degree $2$ surface on $X$ (possibly singular, possibly reducible) which contains the lines $l_1$ and $\sigma$ and contains some line through the point $x_2$.

The upshot of having $1$-twisting surfaces on $X$ of low degree is that the cases we must check for the map $g$ are small in number.  The following is well known.

\begin{lem}
A map $g$ as above can sweep out on $X$ either a non-reduced $2$-plane, the union of two $2$-planes meeting in a line, a conical quadric surface, or a smooth quadric surface.
\end{lem}

We rule out the most degenerate cases first.

\begin{lem}
A map $g$ as above cannot correspond to a non-reduced $2$-plane or the union of two $2$-planes meeting in a line.
\end{lem}

\begin{proof}
If the map $g$ corresponded to a non-reduced plane on $X$, then both $\sigma$ and $l_1$ would be contained in the same $2$-plane.  The lines $\sigma$ and $l_1$ are general lines through $x_1$.  Since $X$ is not linear, this is impossible.

Suppose then that the map $g$ corresponds to the union of two $2$-planes meeting in a line.  By the first paragraph, $l_1$ and $\sigma$ cannot be contained in the same $2$-plane.  Then there must be a plane $P_1$ containing $l_1$ and a plane $P_2$ containing $\sigma$ and they necessarily meet in a line through $x_1$.  This statement may be reinterpreted as saying that on $F(X, x_1)$, the two general points $[l_1]$ and $[\sigma]$ have been connected by a chain of two lines on $F(X, x_1)$.  This will lead to a contradiction.

As every line on $X$ containing $x_1$ is free by Lemma~\ref{lem_generic_properties}, the space $F(X, x_1)$ is a smooth complete intersection of type $(1, \ldots, d_1, \ldots, 1, \ldots, d_c)$ (see Theorem~\ref{thm_numerical_lines}).  The dimension of $F(X, x_1)$ is $n - \textbf{d} - 1$.  The expected dimension of lines on $F(X, x_1)$ through any point is $n - 2 - \frac{1}{2} \sum_{i = 1}^c d_i^2 + d_i$.  As $\sigma$ and $l_1$ are general lines through $x_1$, this is the actual dimension of lines on $F(X, x_1)$ through both $\sigma$ and $l_1$.  The family of lines on $F(X, x_1)$ passing through $\sigma$, respectively through $l_1$, sweeps out a variety of dimension $n - 1 - \frac{1}{2} \sum_{i = 1}^c d_i^2 + d_i$. Then because $2(n - 1 - \frac{1}{2} \sum_{i = 1}^c d_i^2 + d_i) < n - \textbf{d} - 1$, we expect these two varieties are disjoint.

Let $N_2$ be the open subset of the moduli space of two pointed lines on $F(X,x_1)$ where each line is free.  Let $R$ be the fiber product $N_2 \times_{F(X,x_1)} N_2$ where the maps $N_2 \rightarrow F(X, x_1)$ are given by evaluating at the second marked point.  A point of $R$ corresponds to the geometric data of a point $p$ of $F(X,x_1)$, and two free lines through $p$ with a point on each of them.  There is an evaluation map $ev: R \rightarrow F(X, x_1) \times F(X, x_1)$.  Because $R$ only involves free lines on $F(X,x_1)$, the scheme $R$ is smooth and has dimension $n - \textbf{d} - 1 + 2(n - 1 - \frac{1}{2} \sum_{i = 1}^c d_i^2 + d_i)$.  If $ev$ were dominant, then we would have $\dim(R) \geq 2(n - \textbf{d} - 1)$.  This would imply that $n < \sum d_i^2$, a contradiction.  This shows two general points of $F(X, x_1)$ cannot be connected by a chain of two lines on $F(X, x_1)$, and this finishes the claim.
\end{proof}

\begin{prop}
In the above situation, Hypothesis $C$ holds.
\end{prop}

\begin{proof}
With the notation as above, it must be shown that any stable, genus $0$ map $g: B \rightarrow \mc{C}$ with class $\beta$ passing through $f(p)$ and such that $ev(f(\PP^1)) = ev(g(B))$ is unobstructed as a curve on $\mc{C}$ through the point lying over $f(q)$.  By the Lemma above, the data of $g$ is equivalent to either a smooth quadric surface on $X$ or a conical quadric surface on $X$, containing the lines $l_1$ and $\sigma$ and some line $l$ through $x_2$.  Let this pointed line $(x_2 \in l)$ be denoted by the point $v \in \mc{C}$.

If $g$ corresponds to a smooth quadric surface, then it is a map $g: \PP^1 \rightarrow \mc{C}$ which is an embedding and contains $f(p)$, a general point on $\mc{C}$.  In this case $g$ is a free rational curve on $\mc{C}$ and so it is unobstructed while fixing any point, in particular while fixing the point $v$.  Note that we could have simply appealed to Remark~\ref{rmk_weaken_hypothesisC}.

It remains to consider the case where $g$ corresponds to a conical quadric surface on $X$.  Clearly, the conical point must be the intersection of $l_1$ and $\sigma$ and so the point $v$ corresponds to the pointed line $x_2 \in \sigma$.  Unraveling the correspondence, $g$ is a map into $\mc{C}$ from a reducible genus $0$ curve $B$ with two components $B_1$ and $B_2$.  The component $B_1$ maps into the fiber of $ev$; it corresponds to ``swiveling" the lines through the point $x_1$.  The component $B_2$ maps isomorphically to the fiber of $\pi$ over $\sigma \in F(X)$; it corresponds to ``moving" the point $x_1$ to $x_2$ along $\sigma$.  As the point $f(p)$ is general on $\mc{C}$, the curve $B_1$ is free.  Also, $\sigma$ corresponds to a free line on $X$, so it is unobstructed while fixing the point $x_2$.  There is an exact sequence
$$
0 \rightarrow T_{\mc{C}}|_{B_1}(-w) \rightarrow T_{\mc{C}}|_B(-p) \rightarrow T_{\mc{C}}|_{B_2}(-p) \rightarrow 0
$$
where $w$ is the point of $B_1$ mapping to $(x_1 \in \sigma) \in \mc{C}$ and $p$ is the point of $B$ mapping to $v$.  As $B_1$ is free, analyzing the long exact sequence in cohomology gives that $H^1(B, T_{\mc{C}}|_B(-p)) \cong H^1(B_2, T_{\mc{C}}|_{B_2}(-p))$.  As $\mc{C}$ is a $\PP^1$ bundle over $F$ and $B_2$ is a fiber of the map $\pi$, we conclude that $H^1(B_2, T_{\mc{C}}|_{B_2}(-p)) = 0$.  This implies that $B$ is unobstructed while fixing $v$ and the proof is complete.
\end{proof}

\begin{cor}\label{cor_2twist}
For $X$ as in Situation~\ref{sit_ci2}, there is a map $f: \PP^1 \rightarrow \mc{C}$, such that $f^*T_{ev}$ is ample.  Such a map corresponds to a $2$-twisting surface on $X$.
\end{cor}

\begin{cor}
For $X$ as in Situation~\ref{sit_ci}.  Suppose $X$ has a linear-space section which admits $2$ twisting surfaces.  Then $X$ also admits two twisting surfaces.
\end{cor}

\begin{proof}
Let $Y = X \cap \PP^k$.  Suppose $\Sigma \subset Y$ is $2$-twisting.  Abstractly our $\Sigma$ can be written as $\PP(\OO(a) \oplus \OO(a))$ for some $a \geq 2$.  Consider the short exact sequence
$$
0 \rightarrow \mc{N}_Y \rightarrow \mc{N}_X \rightarrow N_{X/Y}|_\Sigma \rightarrow 0
$$
where the sheaf $\mc{N}_X$ denotes the normal sheaf for the map $\Sigma \rightarrow \PP^1 \times X$ (and similarly for $\mc{N}_Y$).  Note that $N_{X/Y}|_\Sigma = \OO_\Sigma(C + aF)$ where $C$ is the class of the square zero section and $F$ is the class of the fiber.  Now twist this sequence down by $\OO(C + 2F)$; the result follows by taking the long exact sequence in cohomology and using the vanishing on $Y$.
\end{proof}

\begin{cor}\label{cor_results}
A general complete intersection $X$ as in Situation~\ref{sit_ci} admits $2$-twisting surfaces.  Every smooth complete intersection $X$ admits $2$-twisting surfaces if $n \geq \Sigma \binom{d_i^2 + d_i - 1}{d_i^2} + d_i^2 - 1$.  Every smooth cubic hypersurface $X \subset \PP^n$ with $n \geq 9$ admits two twisting surfaces.  Every smooth $(2,2)$ complete intersection $X \subset \PP^n$ with $n \geq 8$ admits two twisting surfaces.
\end{cor}
\begin{proof}
The first statement follows because the space of lines on a general smooth complete intersection is smooth.  Setting $m = \Sigma d_i^2$ and choosing a general $\PP^m \subset \PP^n$, then $Y = \PP^m \cap X$ will be smooth and general, and so the result follows from the Theorem.  The second statement follows because a general $\PP^m$ section of any smooth $X$ will be a general complete intersection in $\PP^m$ in this degree range.  This follows from \cite{Starr_Linear_Sections}, Corollary 2.3.  The last two statements follow from the fact that the space of lines a smooth cubic hypersurface (resp. $(2,2)$ complete intersection) is automatically smooth.
\end{proof}

The result is certainly not ideal.  We ask the following question: given $X \subset \PP^n$ a smooth complete intersection, when is the space of lines on $Y = X \cap \PP^m$ smooth for a general $\PP^m \subset \PP^n$?

\section{Strong Rational Simple Connectedness}\label{section_srsc}

We now present the argument that a smooth complete intersection $X$ as in Situation~\ref{sit_ci} which admits $1$ and $2$ twisting surfaces is strongly rationally simply connected.

\begin{defn}
A complete intersection $X \subset \PP^n$ with $\dim(X) \geq 3$ is \textit{strongly rationally simply connected} if for every $m \geq 2$ there is a number $e$ and a canonically defined component of $M_{e,m} \subset \kspacenb{2}{X}{e}$ such that the evaluation map  $ev: M_{e,m} \rightarrow X^m$ is dominant and the general fiber is rationally connected.
\end{defn}

\begin{rmk}
The dimension assumption is in place to ensure that the Picard number of $X$ is $1$.  We then denote the class $e[line]$ by simple, $e$.  The moduli space $\kspace{0}{X}{e}$ of rational curves on such a complete intersection are not known to be irreducible.  Nevertheless, there is always a canonically defined component which behaves functorially (at least in some sense).  Because the space of lines through a general point on $X$ (at least if $X \subset \PP^n$ is a smooth complete intersection satisfying $n - 2 \geq \Sigma d_i$) is irreducible, there is a canonically defined component $M_e \subset \kspace{0}{X}{e}$ as discussed in \cite{dJS} Section 3.  This component will be referred to as the good component; for its properties, see the location cited.  Informally, it is the unique component which parameterizes smoothed out configurations of free curves on $X$, as well as multiple covers of free curves.  Denote by $M_{e,n} \subset \kspace{n}{X}{e}$ the unique component dominating $M_e$.  We will continue with the outline of the proof assuming only the existence of the good component, to remain as general as possible.
\end{rmk}

Let $\alpha$ denote the class in $H^2(X, \ZZ)$ of a line on $X$.  To remain consistent with the notation of \cite{dJS}, a ruled surface $\Sigma$ on $X$ will be said to have $M$-class $(e_1 \cdot \alpha, e_2 \cdot \alpha)$ if the fibers are parameterized by points of the good component $M_{e_1}$ and the section of minimal self intersection (classically, the directrix) is parameterized by a point of the good component $M_{e_2}$.   By Theorem~\ref{thm_one_twist_ci} there is a one twisting surface on $X$ and since both fibers and minimal section are mapped to lines (the space of which is irreducible), then these surfaces have $M$-class $(1 \cdot \alpha, 1 \cdot \alpha)$.  By Lemma~\ref{lem-more-twisting-surfaces}, then, there are $1$-twisting surfaces on $X$ of $M$-class $(e_1 \cdot \alpha, e_2 \cdot \alpha)$ for all $e_1, e_2 \geq 1$.  By Corollary~\ref{cor_2twist}, there is a number $e_0 = e_0(n)$ (possibly large) such that there is a $2$-twisting surface $\Sigma \cong \PP^1 \times \PP^1$ of $M$-class $(1 \cdot \alpha, e_0 \cdot \alpha)$.

To see that this is the case, return to the construction of these surfaces.  They were proved to exist by considering smoothed out configurations of glued together conics in the Fano scheme of lines.  So fibers of the resulting surface are certainly mapped to lines on $X$.  The minimal section class of $\Sigma$ is the ``smoothed out" section class of the conics which correspond to quadric surfaces.  That is, the section class is constructed by smoothing out configurations of free lines on $X$, and so by definition, is in the good component $M_{e_0}$.  By Lemma \ref{lem-more-twisting-surfaces} again, there is a number $e_m$, such that there exist $m$-twisting surfaces of $M$-class $(1 \cdot \alpha, e_m \cdot \alpha)$ on $X$.  From now on we will abbreviate the $M$ type $(c_1 \cdot \alpha, c_2 \cdot \alpha)$ by $(c_1, c_2)$.

\begin{thm}\label{thm_srsc}
Suppose $X$ is a smooth complete intersection of type $(d_1, \ldots, d_c)$ such that the space of lines through a general point of $X$ is irreducible.  Suppose further that $X$ admits one-twisting surfaces of $M$-class $(1,1)$ and $2$ twisting surfaces of $M$ class $(1, e_0)$ for some $e_0$.  Then $X$ is strongly rationally simply connected.
\end{thm}

\begin{cor}
Every complete intersection as in Corollary~\ref{cor_results} is strongly simply rationally connected.
\end{cor}

\begin{rmk}
To reiterate, the conditions of the theorem imply that we can speak of the good component $M_e \subset \kspace{0}{X}{e}$, that 1 twisting surfaces on $X$ of type $(e_1, e_2)$ exist for all $e_1, e_2 \geq 1$, and that there is a number $e_m$ such that $m$-twisting surfaces of type $(1, e_m)$ can be found on $X$.
\end{rmk}

We prove the theorem in a sequence of steps.

Step 1: For each number $e \geq 2$, consider the good component $M_{e\cdot\alpha,2}$ and the restricted evaluation map: $ev: M_{e\cdot\alpha,2} \rightarrow X \times X$.  Then a general fiber is rationally connected.

Proof:
By the above remark, for each integer $k$, there is a 1-twisting surface of type $(1,k)$ on $X$.  The proof that the fiber of $ev_e: M_{e\cdot\alpha,2} \rightarrow X \times X$ over a general point is rationally connected proceeds by induction on $e$.  The base case is $e = 2$ where we can verify directly that the space of lines through two general points on $X$ is cut out by equations of degrees $(1, \ldots, d_1, \ldots, 1 \ldots, d_c, 1, \ldots, d_1 - 1, \ldots, 1, \ldots, d_c - 1)$ in $\PP^n$.  For a generic choice of a point in $X \times X$, this locus will be smooth, non-empty, Fano, and so rationally connected.  We know there exist 1-twisting surfaces of type $(1,1)$.  By Lemma~\ref{lem_sm_point}, the MRC quotient of a strong resolution of a fiber of $M_{2,2} \rightarrow X^2$ is dominated by $\Delta_{1,1}$.  By Lemma~\ref{lem_mrc} then, because a general fiber of $\Delta_{1,1} \rightarrow X^2$ is rationally connected, so is a general fiber of $ev_2$.

By way of induction, assume $e > 2$ and that the result is known for $e - 1$.  We will use the existence of a one twisting surface of type $(1, (e-1))$.  Using the 1-twisting surface machine, we get again by Lemma~\ref{lem_sm_point} that the image of $\Delta_{1,(e-1)}$ intersects the domain of definition of the MRC fibration of a strong resolution of the fiber.  By Lemma~\ref{lem_mrc}, this implies that a general fiber of $ev_e$ is geometrically rationally connected if the fiber of $ev_\Delta : M_\Delta = M_{\alpha, 2} \times M_{(e-1)\alpha, 2} \rightarrow X \times X$ is geometrically rationally connected.

To see the geometric rational connectivity of a general fiber of $ev_\Delta$, consider the projection $p: M_\Delta \rightarrow M_{\alpha, 2}$.  Over the fiber of a general point $(x_1, x_2) \in X^2$, this is exactly the space of pointed lines on $X$ through $x_1$.  As the space of lines through $x_1$ is rationally connected, so too is this space, $F$.  By \cite{GraberHarrisStarr}, to see that the fiber $ev_\Delta^{-1}(x_1, x_2)$ is rationally connected, we can prove that the general fiber of the projection to $F$ is rationally connected.  But this is exactly the space of degree $(e - 1)$ curves (parameterized by the good component $M_{e-1,2}$) passing through $(x_1, x_2)$ and so is rationally connected by induction.  This completes the outline of the proof of rational simple connectedness.

Step 2:  We now proceed by induction on $m$.  Let $m > 2$ and fix an integer $e \geq e_m + 2$.  Set $e' = e + m$, then a general point of a general fiber of $ev|_M: M_{e'\alpha, m} \subset \kspace{m}{X}{e'} \rightarrow X^{m}$ is contained in a rational curve which intersects the boundary $\Delta_{1,(e'-1)}$ in a smooth point.

Proof: This follows directly from Lemma~\ref{lem_sm_point} because $X$ admits $m$ twisting surfaces of type $(1, e)$.

Step 3:  In fact, the general point of a general fiber of $ev|_M$ is contained in a rational curve intersecting the boundary divisor $\Delta_{2, (e'-2)}$ in a smooth point.

Proof: The proof is similar to Lemma~\ref{lem_connect_to_bndry} and Lemma~\ref{lem_sm_point}, and in fact follows from these Lemmas using the more degenerate boundary $\Delta_{1,1,(e'-2)}$.  To be more precise, consider the three evaluation maps $ev_a: M_{e'\cdot\alpha, m} \rightarrow X^{m}$, $ev_b: M_{1\cdot\alpha,2} \times_X M_{(e' - 1)\cdot\alpha, m} \rightarrow X^{m}$, and $ev_c: M_{1\cdot\alpha,2} \times_X M_{1\cdot\alpha,2} \times_X M_{(e' - 2)\cdot\alpha, m} \rightarrow X^{m}$.  The fibers of these evaluation maps over a common general point form a nested triple of varieties.  Lemma~\ref{lem_sm_point} implies that we can connect a general point of a fiber of $ev: M_{e'\cdot\alpha, m+1} \rightarrow X^{m+1}$ to a point of $\Delta_{1,(e'-1)}$ along a rational curve (and similarly for $e' - 1$).  Moreover these boundary points may be taken to be smooth points of the moduli space.  The first fact implies that the MRC quotient of a strong desingularization of fiber of $ev_a$ (resp. $ev_b$) is dominated by the MRC quotient of a strong desingularization of the corresponding fiber of $ev_b$ (resp. $ev_c$).  This is true using the exact same argument found in Lemma~\ref{lem_mrc}.  By transitivity, the MRC quotient strict transform of the fiber of $ev_c$ dominates the strong desingularization of the fiber of $ev_a$. Since the general point of the fiber of $ev_c$ is a smooth point of $M_{e'\cdot\alpha, m+1}$, this MRC quotient is dominated by the transform of $\Delta_{2, e'-2}$.  Less formally, given a reduced curve $C$ in class $|F' + (e' - 2)F|$ attached to two additional fibers, we may first smooth out $C$ and one of the fibers (keeping the attachment point with the other fiber fixed) along a $\PP^1$; then the resulting curve may also be smoothed out along a $\PP^1$ (and all resulting curves may be considered general points of their moduli spaces).  The argument above implies that we may in fact do this along a single rational curve.

Step 4:  The general fiber of $ev_\Delta: M_{2\cdot\alpha, 2} \times_X M_{(e' - 2)\cdot\alpha, m} \rightarrow X \times X^{m-1}$ is rationally connected.

Proof: Fix a general point $(p, (p_1, \ldots, p_{m-1})) \in X \times X^{m-1})$.  By the induction hypothesis, the fiber of $e^{-1}(p_1, \ldots, p_{m-1})$ of $e: M_{(e' - 2)\cdot \alpha, (m-1)} \rightarrow X^{m-1}$ is rationally connected.  We may consider the composition $E: M_{(e' - 2)\cdot\alpha,m} \rightarrow M_{(e' - 2)\cdot\alpha,(m-1)} \rightarrow X^{m-1}$ where the first map forgets the last marked point.  The general fiber of this composition is generically smooth over $e^{-1}((p_1, \ldots, p_{m-1})$.  A conic bundle over a rationally connected variety is rationally connected, so a general fiber of $E$ is also rationally connected.  We then conclude that fiber over $E^{-1}(p_1, \ldots, p_{m-1})$ is rationally connected as well.  The space $ev_\Delta^{-1}(p, (p_1, \ldots, p_{m-1}))$ projects onto $E^{-1}(p_1, \ldots, p_{m-1})$.  The fiber of this projection is the space of conics through $p$ and what may be taken to be a general point on $X$, $p'$.  By Step 1, this space is rationally connected.

Step 5:  A general point of a general fiber of $ev|_M$ is rationally connected.

Proof: By Step 4, a general point of $ev_\Delta: M_{2\cdot\alpha, 2} \times_X M_{(e' - 2)\cdot\alpha, m} \rightarrow X \times X^{m-1}$ is rationally connected.  By Step 3, a general point of the fiber of $ev_\Delta$ can be connected to a general point of $ev|_M$ along a $\PP^1$ (in the fiber).  The proof then follows from another application of Lemma~\ref{lem_mrc} where we take $V$ to be a general fiber of $ev_\Delta$ and $W$ to be the corresponding fiber of $ev|_M$.

\begin{lem} \label{lem_connect_to_bndry}
Suppose $f: \Sigma \rightarrow X$ is an $m$ twisting surface of type $(e_1, e_2)$ and write $e = e_2 + me_1$.  The map $f$ induces a morphism (of stacks):
$$
\kspace{n}{\Sigma}{F' + mF} \rightarrow \kspace{n}{X}{e}.
$$
Points corresponding to reduced divisors $D \in |\mc{O}_\Sigma(F' + mF)|$ with $n$ distinct smooth marked points are smooth points in $\kspace{n}{X}{e}$.  Call $U_{m,m+1}$ the open subset of $\kspace{m+1}{\Sigma}{F' + mF}$ parameterizing smooth divisors in the corresponding curve class with $m+1$ distinct marked points.  Let $M_{m+1}$ be the component of $\kspace{m+1}{X}{e}$ containing the image of $\overline{U_{m,m+1}}$.  A general point of $M_{m+1}$ is contained in a map $g: \PP^1 \rightarrow M_{m+1}$ contained in a fiber of $ev: \kspace{m+1}{X}{e} \rightarrow X^{m+1}$ and intersecting a general point of the image of $\overline{\mc{M}}(\Sigma, \tau_{F, F'})$ where $\tau_{e_1, e_2}$ corresponds to ``combs" consisting of a handle of curve class $F'$, $m$ teeth of class $F$ and one marked point on each tooth and a marked point on the handle.
\end{lem}

\begin{proof}
We outline the proof, see \cite{dJS} Lemma 8.3.  For a reduced divisor $D$ as above, the normal bundle $N_{D/\Sigma}$ is globally generated.  Because $N_{\Sigma/X}$ is globally generated, so is $N_{D/X}$, from which the smoothness statement follows.
Because the surface is $m$-twisting, a general deformation of $D$ is followed by a deformation of $\Sigma$ (on $X$), which remains an $m$ twisting surface.  By a parameter count, a divisor $C$ which is the union of $F'$ and $m$ distinct fibers (and may be taken general in its moduli space) deforms to a smooth divisor in $|F' + mF|$ while fixing $(m + 1)$ points.  Since it is moving in its linear system, we may assume it is doing so along a $\PP^1$.  What's more, the deformation of $D$ may be taken to be a general point of $M_{m+1}$ and we may further assume that a $\PP^1$ connects it to a smooth point of the boundary.
\end{proof}

\begin{lem}
In the Lemma above, we have that $M_{m+1}$ is actually the good component $M_{e \cdot \alpha}$ (here $e = e_2 + me_1$). Moreover, the image of $\overline{\mc{M}}(\Sigma, \tau_{e_1, e_2})$ is contained in the boundary of this good component.
\end{lem}

\begin{proof}
This follows because all curves on $\Sigma$ are free.  Then there is a curve in $M_{m+1}$ whose irreducible components are free, smooth curves parameterized by the good component $M_{e_i \cdot \alpha, 0}$ for various $e_i$.  This characterizes the good component.
\end{proof}

\begin{lem}\label{lem_sm_point}
Assume there exist $m$-twisting surfaces on $X$ of type $(e_1, e_2)$.  Writing $e = e_2 + me_1$, a general point of a general fiber of the evaluation map,
$$
ev: M_{e\cdot\alpha, m+1} \subset \kspace{m+1}{X}{e} \rightarrow X^{m+1}
$$
is contained in a rational curve intersecting the image of the boundary
$$
M_{e_1\cdot\alpha,2} \times_X M_{(e-e_1)\cdot\alpha,m+1} \rightarrow M_{e\cdot\alpha, m+1},
$$
in a smooth point of the fiber.  Thus, the image $\Delta_{e_1, (e - e_1)}$ of the boundary map intersects the domain of definition of the MRC fibration of a strong resolution of the fiber and in particular, dominates the MRC quotient of this resolution.
\end{lem}

\begin{proof}
This follows almost directly from Lemma\ref{lem_connect_to_bndry}.  Indeed, we already know a general point of the fiber of $ev$ is contained in a rational curve intersecting $\mc{M}(\Sigma, \tau_{e_1, e_2})$ in a smooth point.  Recall that this locus parameterizes combs $C$ whose handle $C_0$ has degree $e_2$ and whose $m$ teeth have degree $e_1$ (with the appropriate markings).  We let $B_0$ be one of the teeth, marked additionally at the point of attachment, and we let $B_1$ be the union of the handle and all the other teeth, also marked additionally at this point of attachment.  Then the pair $(B_0, B_1)$ is parameterized by a point of $M_{e_1\cdot\alpha,2} \times_X M_{(e-e_1)\cdot\alpha,m+1}$ and the image in $M_{e\cdot\alpha, m+1}$ is exactly the original comb $C$.  In other words, $\mc{M}(\Sigma, \tau_{e_1, e_2})$ is contained in $\Delta_{e_1, (e - e_1)}$ and so the rational curve from Lemma~\ref{lem_connect_to_bndry} does indeed intersect $\Delta_{e_1, (e - e_1)}$, and it does so in a smooth point.
\end{proof}

\begin{lem}\label{lem_mrc}
Suppose that $V \subset W$ are projective varieties satisfying
i) $V \cap W^{nonsing} \neq \emptyset$.
ii) V is rationally connected.
iii) codim$(V,W) = 1$.
iv) For a general point $v \in V$, there is a $f: \PP^1 \rightarrow W$ such that $f(0) = v$ but $f(\PP^1) \nsupseteq V$.
Then $W$ is also rationally connected.
\end{lem}

\begin{proof}
To prove the Lemma, use the existence of the MRC quotient for a strong resolution $\tilde{W}$ of $W$ which exists by \cite{Kollar-ratcurves}.  This is a rational map $\phi: \tilde{W} \dashrightarrow Q$ such that a general fiber of the map is an equivalence class for the relation ``being connected by a rational curve on $\tilde{W}$".  By definition, there is some open set $U$ of $\tilde{W}$ such that the restriction of $\phi$ to $U$ is regular, proper, and every rational curve in $\tilde{W}$ intersecting $U$ is contained in $U$.  Since the resolution is an isomorphism over the smooth locus of $W$, the strict transform of a rational curve through a generic point of $W_{smooth}$ meeting $V$ now meets $\tilde{V}$.  In other words, a general point of $\tilde{W}$ is contained in a rational curve meeting $\tilde{V}$.  By the preceding remarks then, $\tilde{V}$ meets the generic fiber of $\phi_U$.  That is, $\tilde{V}$ meets $U$ and $\phi_U(U \cap \tilde{V})$ is dense in $Q$.  However, since $V$ is rationally connected, so is $\tilde{V}$ so that $Q$ must be a point.  This implies that $\tilde{W}$ is also rationally connected, so that $W$ is as well..  For a slightly more formal proof, refer to \cite{dJS} Lemmas 8.5 and 8.6.
\end{proof}

We also have the following corollary, a long winded geometric proof of Tsen's Theorem!
\begin{cor}
Let $f: X \rightarrow S$ be a flat, proper morphism to sa smooth, irreducible projective surface.  Assume there is an open $U \subset S$ whose complement has codimension at least $2$ such that $f^{-1}(U)$ is smooth.  Assume that the geometric generic fiber of $f$ is a complete intersections in $n$-dimensional projective space of type $(d_1, \ldots, d_c)$ with $\sum d_i^2 \leq n$, of dimension at least $3$, and that the space of lines on a general fiber of $f$ is smooth.  Let $\mc{L}$ be an $f$-ample invertible sheaf on $f^{-1}(U)$.  Then $f$ admits a rational section.
\end{cor}
\begin{proof}
To apply Theorem~\ref{thm_sections_exist}, we must check that the geometric generic fiber of $f$ is rationally simply connected by chains of free lines and admits a very twisting surface.  The first follows from the proof of Step 1 in Theorem~\ref{thm_srsc}, the second from Corollary~\ref{cor_2twist}.
\end{proof}

\begin{rmk}{A Different Method?}\\
Having created this family of curves $M$ on $ev^{-1}(x)$ for a general $x \in X$, one could apply the method of forming an algebraic quotient as follows:  Take the closure of $M$ in the appropriate Hilbert Scheme, call it $\overline{M}$.  Then we have the following diagram:
$$
\xymatrix{
\overline{C} \ar[d]^\pi \ar[r]^f & ev^{-1}(x)\\
\overline{M} }
$$
Here $\overline{C}$ is the restriction of the universal object over the Hilbert Scheme.  Since both the maps $\pi$ and $f$ are proper, the method explained in Koll\'{a}r's book \cite{Kollar-ratcurves} gives an open set $U \subset ev^{-1}(x)$, a variety $Y$ and a proper map $g: U \rightarrow Y$.  This map has the property that its fibers are equivalence classes under the relation of ``being connected by curves in $\overline{M}$".  In other words, if two points can be connected by a chain of curves from $\overline{M}$ then they map to the same point of $Y$.  The advantage of this method is that the map $g$ is proper, so that all of the curves $C_{y,\sigma, x}$ must be completely contained in $U$ (since they are contracted by $g$ by construction).  The same argument applied above will show that $Y$ must be a point.  Now however, it is not clear how to proceed.  Simply because two general points are connected by chains of curves from $M$, we are not able to (immediately) conclude that given a point $y \in \mc{C}$, then the positive parts of $\phi \in \tilde{M}_{2r}$ passing through $y$ point in ``all directions" at $y$.

We give an example to show what could go wrong.  Suppose $X = \PP^2$ and let $M$ be the set of all lines through a fixed point $p \in X$.  The algebraic method will give that the quotient is a point, because any two points in $X$ can be connected by a chain of curves in $M$ (at most $2$ clearly).  Applying the distribution method though, we will see that the quotient is a $\PP^1$ because at a general point of $X$ the curves in $M$ only point in $1$ direction.  The point $x$ is a singular point for the associated distribution.

It is not clear how to avoid such a situation in the case in which we are interested.  A solution to this problem though, would imply the result for all degree $\underline{d}$ smooth complete intersections in $\PP^{n}$ with $n = \Sigma d_i^2$.
\end{rmk}

\section{Appendix : Twisting Surfaces }\label{app-twisting-surfaces}

The information in this section is all contained in \cite{dJS}.  Due to the central nature it plays in the work above, the main definitions and results are recorded here for convenience/completeness.

All ruled surfaces considered will be maps $\pi: \Sigma \rightarrow \PP^1$ which are well known to be isomorphic to projective bundles $\Sigma \cong \PP(\OO \oplus \OO(-h))$.  The integer $h$ will be called the $h$-type of the ruled surface.  Denote by $F$ the class of a fiber and $E$ the curve class with minimal self intersection $E^2 = -h$.  Denote by $F'$ the divisor class $E + hF$.  It is the unique curve class such that $F'\cdot E = 0$.  For a variety $X \subset \PP^n$ and a map from a ruled surface $\Sigma \rightarrow X$, we have an induced morphism $(\pi, f): \Sigma \rightarrow \PP^1 \times X$.  When $(\pi, f)$ is finite, the vertical normal sheaf $N_{(\pi,f)}$ is defined to be $\text{Coker } T_\Sigma \rightarrow (\pi, f)^* T_{\PP^1 \times X}$.

Suppose there is a ruled surface on $X$ and a curve class in the ruled surface. Given a deformation of the curve in $X$, when is there a deformation of the surface which contains the deformation of the curve?  The following answer motivates the definition of twisting surfaces.

\begin{lem}\label{lem_moving_surfaces}
Suppose $f: \Sigma \rightarrow X$ is a ruled surface as above such that $(\pi,f)$ is finite.  Given a positive integer $n$, suppose further that $N_{(\pi,f)}$ is globally generated and that $H^1(\Sigma, N_{(\pi,f)}(-F' - nF)) = 0$.  Free curves on $\Sigma$ map to free curves on $X$.  Also, for reduced curves $D$ in $|\OO(F' + nF)|$, for every infinitesimal deformation of $D$ in $X$ there is an infinitesimal deformation of $\Sigma$ in $X$ containing the given deformation of $D$.
\end{lem}
\begin{proof}
See \cite{dJS} 7.4.
\end{proof}

Given a twisting surface $f: \Sigma \rightarrow X$, if the pushforward of $F$ maps to the curve class $\beta_1$ and the pushforward of $F'$ maps to $\beta_2$ then we say $f$ has class $(\beta_1, \beta_2)$.  On a complete intersection $X \subseteq \PP^n$ of dimension at least $3$, the curve classes are identified with integers, and this identification will be used.

\begin{rmk}
The work above concerns itself with the existence of twisting surfaces of $h$-type $0$.  These are surfaces which are isomorphic to $\PP^1 \times \PP^1$.  In this case, the minimal curve class is identified with sections of the map $\pi$ and $F' = E$.
\end{rmk}

The existence of 1-twisting surfaces of a given class implies the existence of 1-twisting surfaces of ``larger" class.  Similarly, the existence of 1 and 2 twisting surfaces imply the existence of $m$-twisting surfaces for all integers $m > 0$:

\begin{lem}\label{lem-more-twisting-surfaces}(\cite{dJS} Lemma 7.6 and Corollary 7.7)
1) Suppose that $f: \Sigma \rightarrow X$ is a 1-twisting of $h$-type 0 and class $(a,b)$.  Then for every pair of integers $(d_1, d_2)$ there is a 1-twisting surface of $h$-type 0 and class $(d_1a, d_2b)$.\\
2) Suppose that $f_1: \Sigma_1 \rightarrow X$ is a 1-twisting surface of $h$-type 0 and class $(a,b)$ and $f_2: \Sigma_2 \rightarrow X$ is a 2-twisting surface of $h$-type 0 and class $(a,c)$.  Further, suppose that $f_1$ and $f_2$ map their respective fiber classes to points parameterized by the same irreducible component of $\kspace{0}{X}{a}$.  Then for every positive integer $m$ and every non-negative integer $r$ there exists an $m$-twisting surface $f: \Sigma \rightarrow X$ of $h$-type 0 and class $(a, rb + (m - 1)c)$.  Moreover, the restriction of $f$ to the fiber $F$ parameterizes curves in the same components as $f_1$ and $f_2$ above.
3) For every $1 \leq l \leq n$, every $n$ twisting surface is also $l$ twisting.
\end{lem}
\singlespace
\bibliographystyle{amsalpha}
\bibliography{references}

\end{document}

%% file: preamble.tex
\usepackage{amsmath}
\usepackage{amsthm}
\usepackage{fancyhdr}
\usepackage{amssymb}
\usepackage{enumerate}
\usepackage{amscd}
\usepackage{graphics}
\usepackage{latexsym}
\usepackage{stmaryrd}
\usepackage{verbatim}
\usepackage[all]{xy}

\usepackage{hyperref}

\theoremstyle{plain}
\newtheorem*{introtheorem}{Theorem}
\newtheorem{theorem}{Theorem}[section]
\newtheorem{thm}[theorem]{Theorem}
\newtheorem{cor}[theorem]{Corollary}
\newtheorem{lem}[theorem]{Lemma}
\newtheorem{prop}[theorem]{Proposition}

\theoremstyle{definition}

\newtheorem{defn}[theorem]{Definition}
\newtheorem{cond}[theorem]{Conditions}

\newtheorem{rmk}[theorem]{Remark}
\newtheorem{constr}[theorem]{Construction}
\newtheorem{notat}[theorem]{Notation}

\newtheorem{setup}[theorem]{Setup}
\newtheorem{sit}[theorem]{Situation}

\theoremstyle{remark}

\newcommand{\kspace}[3]{\ensuremath{\overline {\mathcal{M}}_{0,#1}(#2, #3)}}
\newcommand{\kspacenb}[3]{\ensuremath{\mathcal{M}}_{0,#1}(#2, #3)}

\newcommand{\Span}{\text{Span}}

\newcommand{\marpar}[1]{}

\newcommand{\ZZ}{\mathbb{Z}}
\newcommand{\CC}{\mathbb{C}}

\newcommand{\PP}{\mathbb{P}}

\newcommand{\mc}{\mathcal}

\newcommand{\mb}{\mathbf}
\newcommand{\OO}{\mc{O}}

\newcommand{\singlespace}{\renewcommand{\baselinestretch}{1.15} \small \normalsize}